\documentclass[a4paper,11pt, reqno]{amsart}
\usepackage[DIV=12, oneside]{typearea} 
\usepackage[utf8]{inputenc}
\usepackage[T1]{fontenc}
\usepackage[english]{babel}
\usepackage{esint}
\usepackage{amssymb, amsthm, bbm}
\usepackage{enumerate, mathrsfs, hyphenat, graphicx}
\usepackage{thmtools}
\usepackage[colorlinks=true,urlcolor=black, linkcolor=blue,%
citecolor=black, pdfpagelabels]{hyperref}
\theoremstyle{plain}
\declaretheorem[title=Theorem, parent=section]{sa}
\declaretheorem[title=Lemma,sibling=sa]{lem}
\declaretheorem[title=Corollary,sibling=sa]{cor}
\declaretheorem[title=Proposition,sibling=sa]{prop}
\newtheorem*{thm*}{Satz}
\theoremstyle{definition}
\declaretheorem[title=Definition,sibling=sa]{defi}
\declaretheorem[unnumbered,title=Remark]{bem}
\declaretheorem[title=Example]{bsp}
\numberwithin{equation}{section}
\newcommand{\R}{  \mathbb{R}}

\newcommand{\N}{  \mathbb{N}}
\newcommand{\Z}{  \mathbb{Z}}

\newcommand{\cA}{  \mathcal{A}}

\newcommand{\cE}{  \mathcal{E}}

\newcommand{\cK}{  \mathcal{K}}

\newcommand{\cM}{  \mathcal{M}}

\newcommand{\cU}{  \mathcal{U}}

\renewcommand{\epsilon}{\varepsilon}

\DeclareMathOperator{\Supp}{supp}

\newcommand{\norm}[1]{\left\| #1 \right\|}
\newcommand{\bet}[1]{\left| #1 \right|}

\newcommand{\Ind}[1]{\ensuremath \mathbbm{1}_{#1}}

\renewcommand{\d}{\, \textnormal{d}}
\newcommand{\pt}[1]{\partial_t #1}

\newcommand{\akon}{\ensuremath{(2-\alpha)}}
\DeclareMathOperator*{\osc}{osc}
\DeclareMathOperator{\supp}{supp}
\addto\extrasenglish{%

}
\title{Local Regularity for Parabolic Nonlocal Operators}
\author{Matthieu Felsinger, Moritz Kassmann}
\thanks{Both authors were supported by the German Science Foundation DFG via SFB 701}

\makeatletter
\hypersetup{pdftitle={\@title}, pdfauthor={M. Felsinger, M. Kassmann}, pdfkeywords={\@keywords}}
\makeatother
\begin{document}

\begin{abstract}
Weak solutions to parabolic integro-differential operators of order $\alpha \in (\alpha_0, 2)$ are studied. Local a priori estimates of Hölder norms and a weak Harnack inequality are proved. These results are robust with respect to $\alpha \nearrow 2$. In this sense, the presentation is an extension of Moser's result from \cite{moser-neu}.
\end{abstract}

\address{Fakultät für Mathematik\\
Universität Bielefeld \\
Postfach 100131 \\
D-33501 Bielefeld}
\email{m.felsinger@math.uni-bielefeld.de}

\address{Fakultät für Mathematik\\
Universität Bielefeld \\
Postfach 100131 \\
D-33501 Bielefeld}
\email{moritz.kassmann@uni-bielefeld.de}

\date{\today}

\maketitle
\tableofcontents


\section{Introduction}
Throughout this article $\Omega$ denotes a bounded domain in $\R^d$ and $I$ an
open, bounded interval in $\R$.
The aim of this article is to study properties of solutions $u:I
\times \R^{d} \to \R$ to 
\begin{align}\label{eq:par_equation}
 \partial_t u(t,x) - Lu(t,x) = f(t,x), \qquad (t,x) \in I \times \Omega,
\end{align}
where $L$ is an integro-differential operator of the form
\begin{align}\label{eq:def_L}
 L u = \operatorname{p.v.}  \int_{\R^d} \left[ u(t,y)-u(t,x)\right]
k_t(x,y) \d y\, .
\end{align}
The kernel $k \colon \R \times \R^{d} \times \R^d  \to [0,\infty)$, $(t,x,y) \mapsto k_t(x,y)$, is assumed to be measurable with a certain singularity at the diagonal $x=y$.

Note that in the case $k_t(x,y) = \frac{\cA_{d,-\alpha}}{\bet{x-y}^{d+\alpha}}$ with a constant $\cA_{d,-\alpha}$ comparable to $\alpha(2-\alpha)$, the integro-differential operator $L$
defined by  \eqref{eq:def_L} is equal to the pseudo-differential operator $(-\Delta) ^{\alpha/2}$
with symbol $|\xi|^\alpha$. Thus the operator in equation
\eqref{eq:par_equation} can be seen as an integro-differential operator of order
$\alpha$ with bounded measurable coefficients.

Let us specify the class of admissible kernels. We assume that the kernels $k$ are of the form 
$ k_t(x,y)=a(t,x,y) k_0(x,y)$ 
for measurable functions $k_0 \colon \R^d \times \R^d \to [0,\infty)$ and $a
\colon  \R \times \R^{d} \times \R^d  \to [\tfrac12,1]$, which are symmetric
with respect to $x$ and $y$.

Fix $\alpha_0 \in (0,2)$ and $\Lambda \geq \max(1,\alpha_0^{-1})$. We say that a kernel $k$ belongs to $\cK(\alpha_0,\Lambda)$, if there is $\alpha \in (\alpha_0,2)$ such that $k_0$ satisfies the following properties: for every $x_0 \in \R^d$, $\rho \in (0,2)$ and $v \in H^{\alpha/2}(B_{\rho}(x_0))$
\allowdisplaybreaks
\begin{gather*}
\rho^{-2} \int_{\bet{x_0-y}\leq \rho} \bet{x_0-y}^2 k_0(x_0,y) \d y + \int_{\bet{x_0-y}>\rho} k_0(x_0,y) \d y \leq \Lambda \rho^{-\alpha}, \tag{K$_1$} \label{eq:K-1}\\
\begin{split}
 \Lambda^{-1}  \iint\limits_{B\,B} \left[ v(x)-v(y)\right]^2 &k_0(x,y) \d x \d y \leq 
(2-\alpha) \iint\limits_{B\,B}   \frac{\left[ v(x)-v(y)\right]^2}{\bet{x-y}^{d+\alpha}} \d x \d y \\
&\leq \Lambda  \iint\limits_{B\,B}  \left[ v(x)-v(y)\right]^2 k_0(x,y) \d x \d y,\quad \text{where }B=B_{\rho}(x_0).  
\end{split}\tag{K$_2$} \label{eq:K-2}
\end{gather*}

We prove the following two theorems:


\begin{sa}[Weak Harnack inequality] \label{thm:harnack} Let $k \in \cK(\alpha_0,\Lambda)$ for some $\alpha_0 \in (0,2)$ and $\Lambda \geq 1$. Then there is a constant $C=C(d,\alpha_0,\Lambda)$ such that for every supersolution $u$ of \eqref{eq:par_equation} on $Q = (-1,1) \times B_2(0)$ which is nonnegative in $(-1,1) \times \R^d$ the following inequality holds:
\begin{equation*} \label{eq:harnack-speziell}
 \norm{u}_{L^1(U_\ominus)} \leq C \left( \inf_{U_{\oplus}} u + \norm{f}_{L^\infty(Q)} \right) \tag{HI}
\end{equation*}
where $U_\oplus=\left(1-(\tfrac12)^\alpha,1\right) \times B_{1/2}(0)$, $U_\ominus=\left( -1, -1+(\tfrac12)^\alpha \right) \times B_{1/2}(0)$.
\end{sa}


In order to prove Hölder regularity we need an additional assumption on $k_0$:
We say that a kernel $k$ belongs to $\cK'(\alpha_0,\Lambda)$ if $k \in
\cK(\alpha_0, \Lambda)$ and if
\begin{equation*}
\sup_{x \in B_2(0)} \int_{\R^d \setminus B_3(0)} \bet{y}^{1/\Lambda} k_0(x,y) \d y \leq \Lambda\ . \tag{K$_3$} \label{eq:K-3}
\end{equation*}
Note that this condition is satisfied if $ \int_{\R^d \setminus
B_3(0)} \bet{y}^{\delta} k_0(x,y) \d y$ is uniformly bounded in $B_2(0)$ for
some $\delta >0$.

\begin{sa}[Hölder regularity] \label{thm:hoelder_main} Let $k \in \cK'(\alpha_0,\Lambda)$ for some $\alpha_0 \in (0,2)$ and $\Lambda \geq 1$. Then there is a constant $\beta=\beta(d,\alpha_0,\Lambda)$ such that for every solution $u$ of \eqref{eq:par_equation} in $Q=I \times \Omega$ with $f=0$ and every $Q'\Subset Q$ the following estimate holds:
\begin{align}\label{eq:hoelder_main}
 \sup_{(t,x),(s,y) \in Q'} \frac{ \bet{u(t,x)-u(s,y)} }{\left( \bet{x-y} +
\bet{t-s}^{1/\alpha} \right)^{\beta}} \leq \frac{\norm{u}_{L^\infty(I \times
\R^d)}}{\eta^\beta} \,, \tag{HC}
\end{align}
with some constant $\eta=\eta(Q, Q') >0$.
\end{sa}

\begin{bem}{\ }\vspace*{-0.5em}
\begin{enumerate}[1.]\itemsep0.5em
 \item Note that in \autoref{thm:harnack} the domains $U_\oplus$, $U_\ominus$ can be replaced by $\left(\tfrac34 ,1\right) \times B_{1/2}(0)$, $\left(-1,-\tfrac34\right) \times B_{1/2}(0)$, respectively. Similarly, $\bet{t-s}^{1/\alpha}$ can be replaced by $\bet{t-s}^{1/2}$ in \autoref{thm:hoelder_main}.\pagebreak[3]
 \item For this article we choose the most simple characteristic setting in order to explain the main arguments.
\begin{itemize}
 \item[--] One can obtain \autoref{thm:harnack} for solutions $u$ in general domains in $\R^{d+1}$ by rescaling $u$ to a function that is a solution in a standard cylinder $(-1,1) \times B_2(0)$, cf. \autoref{lem:scaling}.
 \item[--] In equation \eqref{eq:par_equation} it is possible to consider more general functions $f$ and additional terms of lower order. When considering terms involving derivatives of the solution $u$, an additional assumption would be $\alpha > 1$. These extensions are analogous to the corresponding modifications in the case of a second order differential operator. 
\end{itemize}
\item Note that a strong Harnack inequality, i.e. $\norm{u}_{L^1(U_\ominus)}$ replaced by $\sup_{U_\ominus} u$ in \eqref{eq:harnack-speziell}, cannot be obtained under our assumptions, see \cite[Theorem 1]{BS05} and the discussion on page 148 there. Thus, the strong formulation of Harnack's inequality fails although conditions \eqref{eq:K-1} and \eqref{eq:K-2} ensure nondegeneracy of the operator $L$ in \eqref{eq:def_L}. In this sense the nonlocal case differs from the case of local diffusion operators.
\end{enumerate}
\end{bem}

One feature of our approach is that the results depend only on $d$, $\alpha_0$ and $\Lambda$, but not on $\alpha$. Thus, in addition to the nonlocality of the operator, one interesting feature is that the order of differentiability can be any number from the interval $(\alpha_0,2)$ and the constants in the results do not depend on this number. We consider it interesting that the classical approach by Moser can be modified to cover this range of problems. However, we do not prove that the classical results of \cite{moser-alt} for the local diffusion case, i.e. $\alpha =2$, follow formally from our results. 

Before we go into details let us shortly discuss related results:
Similar parabolic problems of the above kind are treated already in
\cite{kom95}, where global methods are used. For time-independent kernels the article
\cite{bass-levin} establishes a
Harnack inequality and Hölder regularity estimates with the use of methods both from probability and analysis. The approach of \cite{bass-levin} has been
extended significantly, see \cite{CKK11} as one example. Regularity results like \autoref{thm:hoelder_main} are important for the construction and approximation of corresponding Markov processes. In our case these processes are jump processes with discontinuous paths and without second moments.

Nonlinear nonlocal time-dependent variational problems are studied in \cite{caffarelli-chan-vasseur}. The authors extend the method of De Giorgi to nonlocal parabolic problems. Hölder regularity estimates for linear equations with irregular coefficients are a central tool in this approach. They lead to $C^{1,\beta}$-estimates for problems with translation invariant kernels and finally to existence and regularity of solutions to the nonlinear problem considered. 

In the above mentioned results the constants blow up as $\alpha \nearrow 2$. Robust results have been established for
elliptic equations, e.g. \cite{Kas09}, \cite{dyda-kassmann}, \cite{CaSi11}, \cite{GS12} and recently for fully nonlinear nonlocal parabolic equations in \cite{chang-davila}. See also
\cite{silvestre} for related results in a critical nonlinear setting when $\alpha=1$.

The main features of our contribution can be summarized as follows: Firstly, we prove local regularity results such as a weak Harnack inequality. Secondly, our results are robust for $\alpha \nearrow 2$, i.e. the constants in our main theorems do not depend on $\alpha \in (\alpha_0,2)$. Note that for fixed $\alpha$ both \autoref{thm:harnack} and \autoref{thm:hoelder_main} reflect the intrinsic scaling property of the underlying operator. Thirdly, we allow for a general class of kernels $k_t(x,y)$. In particular, we do not impose pointwise conditions on $k_t(x,y)$ in an essential way. Our article differs from \cite{caffarelli-chan-vasseur} with regard to these three aspects. 

The following two examples illustrate two of these aspects; \autoref{bsp:sequence-of-kernels} illustrates the robustness for $\alpha \nearrow 2$. \autoref{bsp:kegel} shows that $k_t(x,y)$ may be zero on a large set around the diagonal $x=y$.
\begin{bsp}\label{bsp:sequence-of-kernels}
Consider a sequence of kernels $(k^n)_{n \in \N}$ such that $k^n \in \cK(\alpha_0,\Lambda)$ for every $n \in \N$ and some $\alpha_0 \in (0,2)$, $\Lambda \geq 1$ independent of $n \in \N$. For instance $k_t(x,y)$ defined by\footnote{Note that the factor $(2-\alpha_n)$ in \eqref{eq:sequence-of-kernels} is essential to find $\Lambda$ and $\alpha_0$ independent of $n \in \N$.} 
\begin{equation}\label{eq:sequence-of-kernels}
 k^n_t(x,y) = \left(2-\alpha_n\right) \bet{x-y}^{2-\alpha_n} \qquad \text{with } \alpha_n = 2-\tfrac1{n+1}\ 
\end{equation}
belongs to $\cK(1,\Lambda)$ for some $\Lambda=\Lambda(d) \geq 1$.
Let $(u_n)$ be a sequence of solutions to the
corresponding equation \eqref{eq:par_equation}. Then
\eqref{eq:harnack-speziell} holds true for the sequence $(u_n)$ uniformly in $n \in \N$. Furthermore, if $(k^n)$ additionally satisfies \eqref{eq:K-3} uniformly in $n \in \N$ -- such as the kernels in \eqref{eq:sequence-of-kernels} -- then also \eqref{eq:hoelder_main} holds uniformly in $n \in \N$. Note that the theorems are interesting and new
even if $\alpha_n= \alpha$ for some $\alpha \in (0,2)$ and all $n \in \N$.
\end{bsp}
\begin{bsp}\label{bsp:kegel} Fix $\alpha_0 \in (0,2)$. Assume $k_t(x,y)=\frac{2-\alpha}{\bet{x-y}^{d+\alpha}}$ for some $\alpha \in (\alpha_0,2)$. Let $\zeta \in \mathbb{S}^{d-1}$ and $r \in (0,1)$. Set 
\[S= \mathbb{S}^{d-1} \cap \left( B_r(\zeta) \cup B_r(-\zeta) \right) \qquad \text{ and } k_t'(x,y) = k_t(x,y) \mathbbm{1}_{S}(\tfrac{x-y}{\bet{x-y}}).\]
Then we have $k' \in \cK(\alpha_0,\Lambda)$ for some $\Lambda \geq 1$.
\end{bsp}
The article is organized as follows: In \autoref{sec:prelim} we explain the
notion of weak (super-)solutions and discuss the application of this
concept in our setting. In \autoref{sec:mosers_iteration} we perform Moser's
iteration for arbitrary negative exponents of positive supersolutions and for
small positive exponents. \autoref{sec:estimates_log_u} provides estimates
on $\log u$ which are necessary in order to apply the method of Bombieri and
Giusti. The Harnack inequality is proved in \autoref{sec:harnack}. In
\autoref{sec:hoelder} we provide the proof of \autoref{thm:hoelder_main}. In \autoref{sec:A-1} we explain the use of Steklov averages when working with weak solutions of parabolic equations.

\textbf{Acknowledgement:} The authors thank M. Steinhauer and R. Zacher for discussing technical details related to parabolic equations and our approach. We are grateful to an anonymous referee for detailed comments.

\section{Preliminaries}\label{sec:prelim}
The Sobolev space of fractional order $\alpha/2$ is defined by
\begin{align}
H^{\alpha/2}(\Omega) &= \left\{ v \in L^2(\Omega)\colon
\frac{\bet{v(x)-v(y)}}{\bet{x-y}^{(\alpha+d)/2}} \in L^2(\Omega\times \Omega)
\right\}
\intertext{endowed with the norm}
\norm{v}_{H^{\alpha/2}(\Omega)}^2 &= \norm{v}^2_{L^2(\Omega)} + \akon
\iint\limits_{\Omega\, \Omega} \frac{\bet{v(x)-v(y)}^2}{\bet{x-y}^{\alpha+d}} \d x\d y.
\end{align}
We denote by $H^{\alpha/2}_0(\Omega)$ the completion of
$C_c^\infty(\Omega)$
under $\norm{\cdot}_{H^{\alpha/2}(\R^d)}$ and by $H^{-\alpha/2}$ the dual of $H_0^{\alpha/2}$.

By $\inf v$ and $\sup v$ we denote the essential infimum and the essential supremum, respectively, of a given funktion $v$.
\subsection{Concept of weak solutions}
In order to introduce the concept of weak solutions for
(\ref{eq:par_equation}) with $L$ as in (\ref{eq:def_L}) we define a nonlocal
bilinear form associated to $L$ by\footnote{In fact, $\frac12 \cE_t$ is associated to $L$ but we omit the factor $\frac12$ in this work.}
\begin{equation}
 \cE_t(u,v) = \iint\limits_{\R^d\, \R^d} \left[ u(t,y) - u(t,x)\right] \left[
v(t,y)-v(t,x) \right] k_t(x,y) \d x \d y.
\end{equation}
\begin{defi} \label{defi:solution}
Assume $Q=I\times \Omega \subset \R^{d+1}$ and $f \in L^\infty(Q)$. We say that
$u \in L^\infty(I;L^\infty(\R^d))$ is a \emph{supersolution of \eqref{eq:par_equation} in $Q=I \times \Omega$}, if
\begin{enumerate}[(i)]
 \item $u \in C_{loc}(I;L^2_{loc}(\Omega)) \cap L^2_{loc}(I;H^{\alpha/2}_{loc}(\Omega))$,
 \item for every subdomain $\Omega' \Subset \Omega$, for every subinterval $[t_1,t_2]\subset I$ and for every nonnegative test function $\phi \in H^1_{loc}(I;L^2(\Omega')) \cap L^2_{loc}(I;H_0^{\alpha/2}(\Omega'))$,
 \begin{multline} \label{eq:def-solution}
\int_{\Omega'} \phi(t_2,x) u(t_2,x) \d x -\int_{\Omega'} \phi(t_1,x) u(t_1,x) \d x -
\int_{t_1}^{t_2} \int_{\Omega'} u(t,x) \partial_t \phi(t,x) \d x \d t +
\int_{t_1}^{t_2} \cE_t(u,\phi) \d t \\
\geq \int_{t_1}^{t_2} \int_{\Omega'} f(t,x) \phi(t,x) \d x \d t .
\end{multline}
\end{enumerate}
From now on ``$\partial_t u - Lu \geq f$ in $I \times \Omega$`` denotes that $u$ is a supersolution in $I \times \Omega$ in the sense of this definition. Subsolutions and solutions\footnote{In the definition of a solution there is no restriction on the sign of the test function $\phi$.} are defined analogously. 
\end{defi}
The assumptions on $u$ and $\phi$ ensure that all expressions in
\eqref{eq:def-solution} are finite. In order to give sense to $\cE_t(u,\phi)$
the function $\phi$ is extended by $0$. Note that we assume $u$ to be bounded
which can be weakened if takes into account the rate of decay of $k_t(x,y)$ for
large values of $|x-y|$.

If the condition (i) in the \autoref{defi:solution} is weakened by
\[ \text{(i')} \qquad u \in L_{loc}^\infty(I;L^2_{loc}(\Omega)) \cap L^2_{loc}(I;H^{\alpha/2}_{loc}(\Omega)),\]
then one would need to prove $u \in C_{loc}(I;L^2_{loc}(\Omega))$. One possibility is to show that the generalized derivative w.r.t. time of the solution $u$ exists and belongs to $L^2(I';H^{-\alpha/2}(\Omega'))$ for every subinterval $I' \Subset I$ and every subdomain $\Omega' \Subset \Omega$. The conclusion follows from an embedding theorem (for instance \cite[Prop. 23.23]{zeidler}). In the case of (i') the first term in \eqref{eq:def-solution} has to be reinterpreted as
\[ \int_{\Omega'} \phi(t_2,x) u(t_2,x) \d x = \lim_{\Delta t \to 0+} \int_{t_2-\Delta t}^{t_2} \int_{\Omega'} \phi(t,x) u(t,x) \d x \d t \]
and similarly for the second term. 

In the main proofs we do not use \autoref{defi:solution} directly.  The starting point in these proofs is the inequality
\begin{equation*}\label{eq:steklov-solution}
 \int_{\Omega'} \partial_t u(t,x) \phi(t,x) \d x + \cE_t(u(t,\cdot),\phi(t,\cdot)) \geq \int_{\Omega'} f(t,x) \phi(t,x)\d x \qquad \text{for a.e. } t \in I, \tag{\ref*{eq:def-solution}'}
\end{equation*}
where we apply test functions of the form $\phi(t,x)=\psi(x) u^{-q}(t,x)$, $q >0$, where $u$ is a positive supersolution in $I \times \Omega$ and $\psi$ a suitable cut-off function. In particular, we assume that $u$ is a.e. differentiable in time. In \autoref{sec:A-1} we justify this approach.

\subsection{Sobolev and Poincaré inequality}
We provide a Sobolev inequality and a weighted Poincaré inequality for fractional
Sobolev
spaces with constants that are uniform for $\alpha \nearrow 2$:
\begin{prop}[Sobolev inequality]
Let $d \geq 3$ and $\alpha_0 >0$. Then there is a constant $S>0$ such that
for
any $\alpha \in (\alpha_0,2)$, $R>0$, $\sigma= \frac{d}{d-\alpha}$ and $u\in
H^{\alpha/2}(B_R)$ the following
inequality holds:
 \begin{align}\label{eq:sobolev}
  \left( \int_{B_R} \bet{u(x)}^{2 \sigma} \d x
 \right)^{1/\sigma} \leq \akon S \iint\limits_{B_R\, B_R} 
 \frac{\bet{u(x)-u(y)}^2}{\bet{x-y}^{d+\alpha}} \d x \d y + S R^{-\alpha}
\int_{B_R} u^2(x) \d x \,.
\end{align}
\end{prop}
\begin{proof} By adaptation of \cite[Theorem 1]{bourgain} to our situation we
obtain for $v
\in H^{\alpha/2}(B_1(0))$
\[ \norm{v}_{L^{2d/(d-\alpha)}(B_1)}^2 \leq c(d) \frac{1-\alpha/2}{d-\alpha}
\iint\limits_{B_1\, B_1} \frac{\bet{v(x)-v(y)}^2}{\bet{x-y}^{d+\alpha}} \d x \d y +
\norm{v}_{L^2(B_1)}^2, \]
which proves \eqref{eq:sobolev} in the case $R=1$, since $\frac{1}{d-\alpha}
\leq (d-2)^{-1}$. The result for general $R>0$ follows after a change of variables.
\end{proof}

\begin{bem}
We provide the proof of \autoref{thm:harnack} only for the case $d \geq 3$.
The assertion stays true if $d \in \{1,2\}$. Tn this case one would use a
Sobolev inequality of the form \eqref{eq:sobolev} with
$\sigma=\frac{3}{3-\alpha}$ which is true for a bounded range of radii $R$.
Moreover, one would need to set $\kappa = 1 + \frac{\alpha}{3}$ in
\autoref{prop:step-neg} and \autoref{prop:step-pos}.
\end{bem}

In order to derive estimates on $\log u$ in \autoref{sec:estimates_log_u} we
will need a weighted Poincaré inequality of fractional order on $B_{3/2}$.

\begin{prop}[Weighted Poincaré inequality] \label{lem:poincare} Let $\Psi \colon {B_{3/2}} \to [0,1]$ be defined by $\Psi(x)=(\tfrac32-\bet{x}) \wedge 1$ and $k \in \cK(\alpha_0,\Lambda)$ for some $\alpha_0 \in (0,2)$ and $\Lambda\geq 1$. Then there is a positive constant $c_2(d,\alpha_0, \Lambda)$ such that for every $v \in L^1(B_{3/2}, \Psi(x)\d x)$
\begin{equation*}
 \int_{B_{3/2}} \left[ v(x)-v_{\Psi} \right]^2 \Psi(x) \d x 
 \leq c_2 \iint\limits_{B_{3/2}\, B_{3/2}} \left[ v(x)-v(y)\right]^2 k_t(x,y) \left( \Psi(x) \wedge \Psi(y) \right) \d x \d y\, ,
\end{equation*}
where $\displaystyle v_{\Psi}=  \Bigl(\int_{B_{3/2}} \Psi(x)\d x\Bigr)^{-1} \int_{B_{3/2}} v(x) \Psi(x)\d x$.
\end{prop}
\begin{proof}
For $ x\in B_{3/2} \setminus \overline{B_1}$ write
$\Psi(x) = 2 \int_{1}^{3/2} \Ind{B_s}(x) \d s$. Then for some $\alpha \in (\alpha_0,2)$
\allowdisplaybreaks
\begin{align*}
 \int_{B_{3/2}} \int_{B_{3/2}}& \left[ v(x)-v(y)\right]^2 k_t(x,y) \left( \Psi(x) \wedge \Psi(y) \right) \d x \d y \\
 &= \int_{B_{3/2}} \int_{B_{3/2}} \left[ v(x)-v(y)\right]^2 k_t(x,y) 2 \int_{1}^{3/2} \Ind{B_s}(x) \Ind{B_s}(y) \d s \d x \d y \\
 &= 2 \int_{1}^{3/2} \int_{B_{3/2}} \int_{B_{3/2}} \left[ v(x)-v(y)\right]^2 k_t(x,y) \Ind{B_s}(x) \Ind{B_s}(y) \d x \d y \d s \\
 &\geq 2 \Lambda^{-1} (2-\alpha) \int_{1}^{3/2} \int_{B_{3/2}} \int_{B_{3/2}} \frac{\left[ v(x)-v(y)\right]^2}{\bet{x-y}^{d+\alpha}}  \Ind{B_s}(x) \Ind{B_s}(y) \d x \d y \d s \\
 &= \Lambda^{-1} (2-\alpha) \int_{B_{3/2}} \int_{B_{3/2}} \frac{\left[ v(x)-v(y)\right]^2}{\bet{x-y}^{d+\alpha}} \left( \Psi(x) \wedge \Psi(y) \right) \d x \d y\ ,
\end{align*}
where we have applied \eqref{eq:K-2} to obtain the inequality. The assertion of \autoref{lem:poincare} follows now immediately from \cite[Corollary 6]{DK12}.
\end{proof}

\subsection{Scaling property and standard cylindrical domains}
Let us briefly explain the scaling behaviour of equation
(\ref{eq:par_equation}). Here and later in the article we will use the following
notation. Define
for $r>0$ $B_r(x_0) =\left\{ x \in \R^d \colon \bet{x-x_0} < r \right\}$ and
\begin{align*}
I_r(t_0)&=(t_0-r^\alpha,t_0+r^\alpha),&Q_r(x_0,t_0)&= I_r(t_0) \times B_r(x_0),&\\
I_{\oplus}(r) &= (0,r^\alpha),& Q_\oplus(r) &= I_{\oplus}(r) \times B_r(0),\\
I_{\ominus}(r) &= (-r^\alpha,0),& Q_\ominus(r) &= I_{\ominus}(r) \times B_r(0).
\end{align*}
\begin{figure}[h]
\centering
\includegraphics{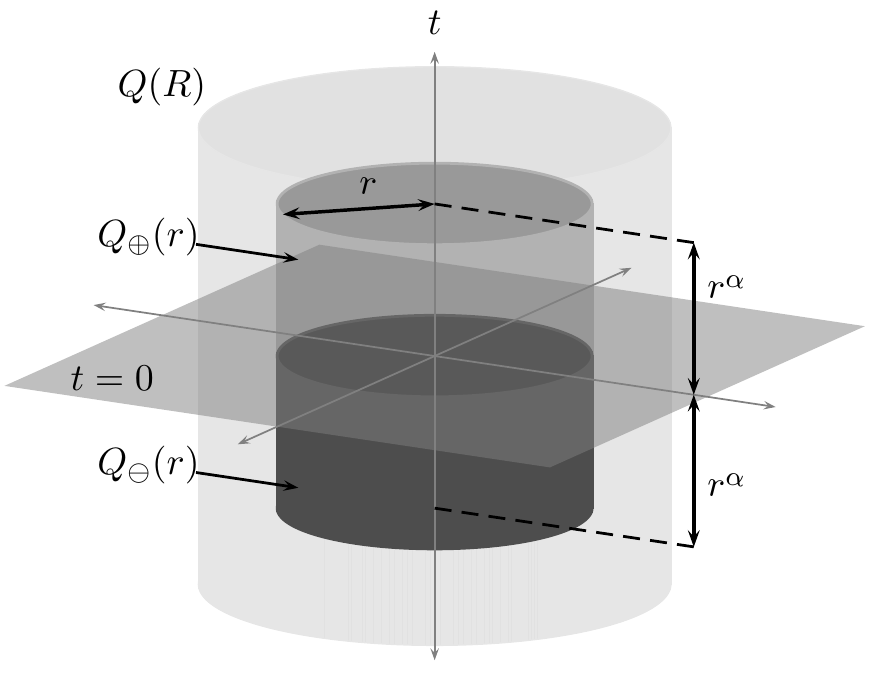}
\caption{Standard cylindrical domains}
\end{figure}
\begin{lem}[Scaling property] \label{lem:scaling} 
Fix $\alpha_0 \in (0,2)$, $\xi \in \R^d$, $\tau \in \R$ and $r>0$. Assume that there is $\alpha \in (\alpha_0,2)$ and $\Lambda \geq \max(1,\alpha_0^{-1})$ such that the kernel $k_t(x,y)=a(t,x,y) k_0(x,y)$ satisfies the following properties: for every $x_0 \in \R^d$, $\rho \in (0,2r)$ and $v \in H^{\alpha/2}(B_{\rho}(x_0))$
\begin{align*}
 \rho^{-2} \int_{\bet{x_0-y}\leq \rho} \bet{x_0-y}^2  k_0(x_0,y) \d y + \int_{\bet{x_0-y}>\rho}  k_0(x_0,y) \d y &\leq \Lambda \rho^{-\alpha},\\
\Lambda^{-1}  \iint\limits_{B\,B} \left[ v(x)-v(y)\right]^2  k_0(x,y) \d x \d y &\leq 
 (2-\alpha) \iint\limits_{B\,B}   \frac{\left[ v(x)-v(y)\right]^2}{\bet{x-y}^{d+\alpha}} \d x \d y \\
  \leq \Lambda  \iint\limits_{B\,B}   \left[ v(x)-v(y)\right]^2 &k_0(x,y) \d x \d y,\quad \text{where }B=B_{\rho}(x_0),\\
\sup_{x \in B_{2r}(\xi)} \int_{\R^d\setminus B_{3r}(\xi)} \bet{y}^{1/\Lambda}  k_0(x,y) \d y &\leq \Lambda\ r^{1/\Lambda-\alpha}.
\end{align*}
Let $u$ be a supersolution of the corresponding equation \eqref{eq:par_equation} in $Q \Supset Q_r(\xi,\tau)$. \\ Then $\widetilde u(t,x)=u(r^\alpha t+\tau,rx+\xi)$ satisfies the following inequality for every nonnegative \\
 $\phi \in H^1((-1,1);L^2(B_1(0))) \cap L^2((-1,1);H_0^{\alpha/2}(B_1(0)))$
 \begin{multline*}  \int_{B_1} \phi(t,x) \widetilde u(t,x) \d x \Bigr|_{t=-1}^1 -
\int_{Q_1(0)} \widetilde u(t,x) \partial_t \phi(t,x) \d x \d t\\ +
\int_{-1}^1 \iint\limits_{\R^d\, \R^d} \left[ \widetilde u(t,x)- \widetilde u(t,y)\right] \left[ \phi(t,x)-\phi(t,y)\right] \widetilde k_t(x,y) 
\d x \d y \d t 
 \geq \int_{Q_1(0)} r^\alpha \widetilde f(t,x) \phi(t,x)\, \d x \d t,
 \end{multline*}
where $\widetilde f(t,x)=f(t r^\alpha+\tau,r x+\xi)$ and 
\[ \widetilde k_t(x,y)=a(r^\alpha t+\tau,rx+\xi,ry+\xi)\, r^{d+\alpha}\, k_0(rx+\xi,ry+\xi).\] In particular, $\widetilde k$ belongs to $\cK'(\alpha_0,\Lambda)$.
\end{lem}
%
\section{Moser's iteration for positive
supersolutions}\label{sec:mosers_iteration}

In this section we provide local estimates of negative powers and
small positive powers of supersolutions $u$. In the first subsection we
collect computation rules which we need. Then we prove the basic step in
Moser's iteration process for negative powers which leads to a first estimate
of $\inf\limits u$. Last we estimate small positive powers of positive
supersolutions $u$. 

Throughout this section we assume that the kernel $k$ in \eqref{eq:par_equation} belongs to some $\cK(\alpha_0,\Lambda)$.
\subsection{Some algebraic inequalities}
The second-named author thanks I. Popescu for the following proposition.
\begin{prop} \label{prop:mittelwert}
Let $f,g \in C^1([a,b])$. Then 
\begin{equation} \label{eq:mittelwert}
 \frac{f(b)-f(a)}{b-a}+ \left( \frac{g(b)-g(a)}{b-a} \right)^2 \leq \max_{t \in
[a,b]} \left[ f'(t) + (g'(t))^2 \right].
\end{equation}
\end{prop}
\begin{proof}
Assume that \eqref{eq:mittelwert} was not true and integrate the reversed
inequality over $[a,b]$ resulting in 
\[ f(b) - f(a) + \frac{ (g(b)-g(a))^2}{b-a} > f(b)-f(a) + \int_a^b (g'(t))^2 \d
t,\]
which is equivalent to
\[ \left( \frac{ g(b)-g(a)}{(b-a)}\right)^2 > \frac{1}{b-a} \int_a^b (g'(t))^2
\d t.\]
This is a contradiction (Jensen's inequality) and hence
\autoref{prop:mittelwert} is proved.
\end{proof}
\begin{lem} \label{lem:convex}
Let $q>0, q \neq 1$ and $a,b > 0$. Then
\begin{equation} \label{eq:q-ineq}
 (b-a)(a^{-q}-b^{-q}) \geq \frac{4q}{(1-q)^2} \left(
b^{\frac{1-q}{2}}-a^{\frac{1-q}{2}}\right)^2.
\end{equation}
\end{lem}
\begin{proof}
Setting $c(q)=\frac{(1-q)^2}{4q}$, \eqref{eq:q-ineq} is equivalent to
\[ c(q) \frac{b^{-q}-a^{-q}}{b-a} + \frac{ \left(
b^{\frac{1-q}{2}}-a^{\frac{1-q}{2}}\right)^2}{(b-a)^2} \leq 0.\]
\autoref{prop:mittelwert} with
$f(t)=c(q) t^{-q}$ and $g(t)=t^{\frac{1-q}{2}}$ yields
\[  c(q) \frac{b^{-q}-a^{-q}}{b-a} + \frac{ \left(
b^{\frac{1-q}{2}}-a^{\frac{1-q}{2}}\right)^2}{(b-a)^2} \leq 
 \max_{t \in [a,b]} t^{-1-q} \left( -\frac{(1-q)^2}{4} + \frac{(1-q)^2}{4}
\right) = 0,
\]
which proves inequality \eqref{eq:q-ineq}.
\end{proof}
Part (i) of the following lemma is taken from \cite[Lemma
2.5]{Kas09}.
\begin{lem}\label{lem:guelle}{\ } 
\begin{enumerate}[(i)]
 \item Let $q>1$, $a,b > 0$ and $\tau_1, \tau_2 \geq 0$. Set $\vartheta(q) =\max\left\{ 4, \frac{6q-5}{2} \right\}$. Then
\begin{multline}\label{eq:guelle-1}
 (b-a) \left( \tau_1^{q+1} a^{-q} - \tau_2^{q+1} b^{-q} \right) \geq
\frac{1}{q-1} (\tau_1 \tau_2) \left[ \left(
\frac{b}{\tau_2}\right)^{\frac{1-q}{2}} - \left(
\frac{a}{\tau_1}\right)^{\frac{1-q}{2}} \right]^2 \\
 - \vartheta(q) \left( \tau_2 -\tau_1\right)^2
\left[ \left( \frac{b}{\tau_2}\right)^{1-q} + \left(
\frac{a}{\tau_1}\right)^{1-q} \right].
\end{multline}
Since $1-q<0$ the division by $\tau_1 = 0$ or $\tau_2 = 0$ is allowed. 
\item Let $q\in (0,1)$, $a,b > 0$ and $\tau_1, \tau_2 \geq 0$. Set $\zeta(q)=\frac{4q}{1-q}$, $\zeta_1(q)=\frac16 \zeta(q)$ and $\zeta_2(q)=\zeta(q)+ \frac{9}{q}$. Then
\begin{multline}\label{eq:guelle-2}
(b-a)\left( \tau_1^2 a^{-q} - \tau_2^2 b^{-q} \right) \geq \zeta_1(q) \left( \tau_2 b^{\frac{1-q}{2}} - \tau_1 a^{\frac{1-q}{2}} \right)^2 
- \zeta_2(q) (\tau_2-\tau_1)^2 \left( b^{1-q} + a^{1-q} \right)
\end{multline}
\end{enumerate}
\end{lem}
\begin{proof}
Here we only prove \eqref{eq:guelle-2}; for the proof of \eqref{eq:guelle-1} we refer to \cite[pp. 5-6]{Kas09}. \eqref{eq:guelle-2} is easily checked if $\tau_2 =0$. If $\tau_1=0$ and $\tau_2>0$ the inequality reads
\[ -b^{1-q} + ab^{-q} \geq \left( \zeta_1(q)-\zeta_2(q) \right) b^{1-q} - \zeta_2(q)a^{1-q}.\]
This is true since $ \zeta_1(q)-\zeta_2(q) < -1$.

Now we consider the case $\tau_1 \tau_2>0$. We can assume $b \geq a$ due to symmetry. Setting $t=\frac{b}{a} \geq 1$, $s = \frac{\tau_2}{\tau_1} > 0$ and $\lambda = s^2 t^{-q}$, assertion \eqref{eq:guelle-2} is equivalent to
\begin{equation}\label{eq:guelle-lambda}
 \zeta_1(q) \left(\sqrt{\lambda t} - 1\right)^2 \leq (t-1)(1-\lambda) + \zeta_2(q) (s-1)^2\left(t^{1-q}+1 \right)
\end{equation}
We estimate
\begin{align*}
 \left(\sqrt{\lambda t} - 1\right)^2 &\leq 2 \left( \sqrt{\lambda t} - t^{\frac{1-q}{2}} \right)^2 + 2 \left( t^{\frac{1-q}{2}} - 1\right)^2 = 2(s-1)^2 t^{1-q} + 2 \left( t^{\frac{1-q}{2}} - 1\right)^2 \\
&\leq 2(s-1)^2 t^{1-q} + \frac{2}{\zeta(q)} (t-1)(1-t^{-q}),
\end{align*}
where we have used \autoref{lem:convex} in the last inequality noting that $\frac{4q}{(1-q)^2} \geq \frac{4q}{1-q} =  \zeta(q)$ for $q \in (0,1)$. We decompose the last factor of the above inequality as follows:
\[ 1-t^{-q} = (1-\lambda) + (\lambda - t^{-q}) = (1-\lambda) + (s-1)^2 t^{-q} + 2(s-1) t^{-q} .\]
This implies
\begin{equation}\label{eq:decomposition}
 \left(\sqrt{\lambda t} - 1\right)^2 \leq \left( 2 + \frac{2}{\zeta(q)} \right) (s-1)^2 t^{1-q} + \frac{2}{\zeta(q)} (t-1)(1-\lambda) + \frac{4}{\zeta(q)} (t-1) (s-1) t^{-q}.
\end{equation}
It remains to estimate the last term in \eqref{eq:decomposition}. To this end we consider different ranges of $t \in [1,\infty)$ and $s \in (0,\infty)$:
\begin{enumerate}[a)]
 \item $t >1$, $s \in (1,2)$ and $t-1 > \frac4q t (s-1)$: By the mean value theorem, there is $\xi \in (1,t)$ such that $t^q-1=q \xi^{q-1}(t-1)$. Then we can estimate
\[ \frac{(s+2)(s-1)}{t-1} \leq \frac{q(s+2)}{4 t} \leq \frac{q}{t} \leq \frac{q}{t^{1-q}} \leq q \xi^{q-1} = \frac{t^q-1}{t-1}.\]
Therefore
\[ (s+2)(s-1) \leq t^q-1, \quad \text{or equivalently} \quad s-1 \leq t^q-s^2=t^q(1-\lambda).\]
This implies $(t-1) (s-1) t^{-q} \leq (t-1)(1-\lambda)$. We deduce from \eqref{eq:decomposition}
\begin{subequations}
 \begin{equation} \label{eq:cases-1}
 \tfrac16 \zeta(q) \left(\sqrt{\lambda t} - 1\right)^2 \leq \left( \tfrac13 \zeta(q) + \tfrac13 \right) (s-1)^2 t^{1-q} + (t-1)(1-\lambda).
\end{equation}
\item $t >1$, $s \in (1,2)$ and $t-1 \leq \frac4q t (s-1)$: In this case $(t-1) (s-1) t^{-q} \leq \tfrac4q t^{1-q} (s-1)^2$ and -- again by \eqref{eq:decomposition} --
\begin{equation}\label{eq:cases-2}
 \tfrac12 \zeta(q) \left(\sqrt{\lambda t} - 1\right)^2 \leq \left( \zeta(q) + 1+\tfrac{8}{q}\right) (s-1)^2 t^{1-q} + (t-1)(1-\lambda).
\end{equation}
\item $t=1$ or $s \leq 1$: Then obviously $(t-1) (s-1) t^{-q} \leq 0$ and
\begin{equation}\label{eq:cases-3}
 \tfrac12 \zeta(q) \left(\sqrt{\lambda t} - 1\right)^2 \leq \left( \zeta(q) + 1 \right) (s-1)^2 t^{1-q} + (t-1)(1-\lambda).
\end{equation}
\item $t>1$, $s \geq 2$:  Using $s-1 \leq (s-1)^2$ we obtain $(t-1) (s-1) t^{-q} \leq (s-1)^2 t^{1-q}$ and
\begin{equation}\label{eq:cases-4}
 \tfrac12 \zeta(q) \left(\sqrt{\lambda t} - 1\right)^2 \leq \left( \zeta(q) + 3\right) (s-1)^2 t^{1-q} + (t-1)(1-\lambda).
\end{equation}
\end{subequations}
\end{enumerate}
Combining inequalities \eqref{eq:cases-1}-\eqref{eq:cases-4} we obtain \eqref{eq:guelle-lambda} since $3 < 1 +\tfrac8q < \tfrac9q$. This finishes the proof of \autoref{lem:guelle}.
\end{proof}

\subsection{Basic step for negative exponents}
The following proposition provides the elementary step of Moser's iteration. Its
proof imitates Moser's ideas in \cite{moser-alt} and \cite{moser-neu},
respectively.
\begin{prop}\label{prop:step-neg} Let $\frac12 \leq r <R \leq 1$ and $p>0$. Then every
nonnegative supersolution $u$ in $Q = I \times \Omega$, $Q \Supset Q_\ominus(R)$, with $u \geq \epsilon>0$ in $Q$
satisfies the following inequality
\begin{align} 
  \left( \int_{Q_\ominus(r)} \widetilde u^{-\kappa p}(t,x) \d
x \d t \right)^{1/\kappa} &\leq A  \int_{Q_\ominus(R)} \widetilde u^{-p}(t,x) \d
x \d t\ ,\label{eq:moment-neg}
\end{align}
where $\widetilde u = u + \norm{f}_{L^\infty(Q)}$, $\kappa=1+\frac{\alpha}{d}$
and $A$ can be chosen as
\begin{equation} \label{eq:const-neg}
 A = C(p+1)^2 \left( \left(R-r\right)^{-\alpha}+ (R^{\alpha}-r^\alpha)^{-1}
\right) \qquad \text{with } C=C(d,\alpha_0,\Lambda).
\end{equation}
\end{prop}
\begin{bem}
Note that
\[ \frac{1}{\left(R-r\right)^{\alpha}}+ \frac{1}{(R^{\alpha}-r^\alpha)} \leq
 \begin{cases}
     \frac{2}{\left(R-r\right)^{\alpha}}\quad &\text{ for } \alpha \in [1,2),\\
    \frac{2}{(R^{\alpha}-r^\alpha)}&\text{ for } \alpha \in (0,1].
   \end{cases}
\]
\end{bem}
\begin{proof}
Let $u$ be a supersolution in $Q$ with $u \geq \epsilon >0$ in $Q$. We set $\widetilde u =
u+\norm{f}_{L^\infty(Q)}$. If $f=0$ a.e. in $Q$ we set $\widetilde u = u+
\delta$ and pass to the limit $\delta \searrow 0$ in the end.  For $q >1$ define
\[ v(t,x) = \widetilde u^{\frac{1-q}{2}}(t,x), \qquad \phi(t,x) = \widetilde
u^{-q}(t,x) \psi^{q+1}(x),\]
where $\psi\colon \R^d \to [0,1]$ is defined by $\psi(x)= \left(
\frac{R-\bet{x}}{R-r} \wedge 1 \right) \vee 0$. Obviously, $\psi^{q+1} \in
H_0^{\alpha/2}(B_R)$.
Hence, from \eqref{eq:steklov-solution} we obtain\enlargethispage{2em}
\allowdisplaybreaks
 \begin{multline} \int_{B_R} -\psi^{q+1}(x) \widetilde u^{-q}(t,x) \pt
\widetilde u(t,x) \d x+\\ + \iint\limits_{\R^d\, \R^d} \left[ \widetilde
u(t,x)-\widetilde u(t,y) \right] \left[ \psi^{q+1}(y) \widetilde u^{-q}(t,y) -
\psi^{q+1}(x)\widetilde u^{-q}(t,x)\right] k_t(x,y) \d x \d y\\ \leq \int_{B_R} -
\psi^{q+1}(x) \widetilde u^{-q}(t,x) f(t,x) \d x.
 \nonumber
\end{multline}
Applying \autoref{lem:guelle}(i) (remember the definition of
$\vartheta(q)$ therein) and rewriting \\ $\pt{v^2} =
(1-q)  \widetilde u^{-q} \pt{\widetilde u}$ yields
\begin{multline}
\tfrac1{q-1} \int_{B_R} \psi^{q+1} \pt{(v^2)} \d x+ \tfrac{1}{q-1}\iint\limits_{\R^d\, \R^d}  \psi(x)\psi(y) \left[ \left( \tfrac{\widetilde u(t,x)}{\psi(x)}
\right)^{\frac{1-q}{2}} - \left( \tfrac{\widetilde u(t,y)}{\psi(y)}
\right)^{\frac{1-q}{2}} \right]^2 k_t(x,y) \d x \d y \\
\leq \vartheta(q) \iint\limits_{\R^d\, \R^d} \left[ \psi(x)-\psi(y) \right]^2
\left[ \left(\frac{\widetilde u(t,x)}{\psi(x)} \right)^{1-q} +
\left(\frac{\widetilde u(t,y)}{\psi(y)} \right)^{1-q} \right] k_t(x,y) \d x \d
y\\ + \int_{B_R}  \psi^{q+1}(x) \bet{\widetilde u^{-q}(t,x)} \bet{f(t,x)} \d x\, .
\label{eq:int}
\end{multline}
The definition of $\psi$ results in the two estimates
\begin{align}
\begin{split} \label{eq:3-4}
\iint\limits_{\R^d\, \R^d} \left[\psi(x)\psi(y)\right] \left[ \left(
\frac{\widetilde u(t,x)}{\psi(x)} \right)^{\frac{1-q}{2}} - \left(
\frac{\widetilde u(t,y)}{\psi(y)} \right)^{\frac{1-q}{2}} \right]^2 k_t(x,y) \d
x \d y \\
\geq \iint\limits_{B_{r}\, B_r} \left[ v(t,x) - v(t,y) \right]^2 k_t(x,y) \d x \d
y,
\end{split}
\intertext{and}
\label{eq:3-5}
\iint\limits_{\R^d\, \R^d} &\left[ \psi(x)-\psi(y) \right]^2 \left[
\left(\frac{\widetilde u(t,x)}{\psi(x)} \right)^{1-q} + \left(\frac{\widetilde
u(t,y)}{\psi(y)} \right)^{1-q} \right] k_t(x,y) \d x \d y \nonumber \\
&\leq 2 \iint\limits_{B_{R}\, B_R} \left[ \psi(x)-\psi(y) \right]^2 \widetilde
u^{1-q}(t,x)\, k_t(x,y) \d x \d y  \nonumber \\
&\qquad \qquad + 4 \iint\limits_{B_{R}\, B_R^c} \left[ \psi(x)-\psi(y)\right]^2 \widetilde u^{1-q}(t,x) k_t(x,y) \d y \d x\nonumber \\
& \leq c_1(d,\Lambda) (R-r)^{-\alpha} \int_{B_R} v^2(x) \d x,
\end{align}
where we have used \eqref{eq:K-1} and $\sup_{x,y\in \R^d} \frac{\bet{\psi(x)-\psi(y)}^2}{\bet{x-y}^2} \leq
\frac{1}{(R-r)^2}$. Combining \eqref{eq:3-4}, \eqref{eq:3-5} and the fact that
$\norm{f/\widetilde u}_{L^{\infty}(Q)} \leq 1$ we obtain from \eqref{eq:int}
\begin{multline} \label{eq:raum-absch}
\int_{B_R} \psi^{q+1}(x) \pt{(v^2)}(t,x) \d x+   \iint\limits_{B_{r}\, B_r} \left[ v(t,x) - v(t,y) \right]^2 k_t(x,y) \d x \d y \\
\leq \ (q-1) \left(1+ \vartheta(q) c_1 (R-r)^{-\alpha} \right) \int_{B_R} v^2(x)
\d x.
\end{multline}
Now define a piecewise differentiable function $\chi \colon \R \to [0,1]$ by
$\chi(t)= \left( \frac{t+R^\alpha}{R^\alpha-r^\alpha} \wedge 1 \right) \vee 0$.
Multiplying \eqref{eq:raum-absch} by $\chi^2$ we get\enlargethispage{1em}
\begin{multline}
\pt \int_{B_R} \psi^{q+1}(x) \left[\chi(t) v(t,x)\right]^2 \d x+
  \chi^2(t) \iint\limits_{B_{r}\, B_r}  \left[ v(t,x) - v(t,y)
\right]^2 k_t(x,y) \d x \d y \\
\leq c_2 (q-1) \vartheta(q) (R-r)^{-\alpha} \chi^2(t) 
\int_{B_R} v^2(t,x) \d x + 2 \chi(t) \bet{\chi'(t)} \int_{B_R} v^2(t,x)
\end{multline}
and integrating this inequality from $-R^{\alpha}$ to some $t \in I_{\ominus}(r)$
yields
\begin{multline*}
\int_{B_R} \psi^{q+1}(x) (\chi(t) v(t,x))^2 \d x +  
\int_{-R^\alpha}^t \chi^2(s) \iint\limits_{B_{r}\, B_r} \left[ v(s,x) - v(s,y)
\right]^2 k_s(x,y) \d x \d y \d s \\
\leq c_2 (q-1) \vartheta(q) (R-r)^{-\alpha} 
\int_{-R^\alpha}^t \chi^2(s)  \int_{B_R} v^2(s,x) \d x \d s\ +\\
+ \int_{-R^\alpha}^t 2 \chi(s) \bet{\chi'(s)} \int_{B_R} v^2(s,x) \d x \d s,
\end{multline*}
which implies, noting that $\bet{\chi'} \leq \frac{1}{R^\alpha-r^\alpha}$,
\begin{multline} \label{eq:cacciopoli}
\sup_{t \in I_{\ominus}(r)} \int_{B_{r}} v^2(t,x) \d x +  
\int_{Q_{\ominus}(r)} \int_{B_{r}} \left[ v(s,x) - v(s,y) \right]^2 k_s(x,y) \d
x \d y \d s \\
\leq c_3 (q-1) \vartheta(q) \left( (R-r)^{-\alpha} + (R^{\alpha}-r^\alpha)^{-1} \right) \int_{Q_{\ominus}(R)} v^2(s,x) \d
x \d s.
\end{multline}

In order to estimate the second term on the left-hand side from below we apply Hölder's inequality with exponents
$\theta=\frac{d}{d-\alpha}$, $\theta'=\frac{d}{\alpha}$ to the integrand
$v^{2\kappa}$ and then we make use of Sobolev's inequality \eqref{eq:sobolev}:
\begin{align*}
\int_{Q_\ominus(r)} &v^{2\kappa}(t,x) \d x \d t = \int_{Q_\ominus(r)} v^{2}(t,x)
v^{2\alpha/d}(t,x) \d x \d t \\
 &\leq \int_{I_\ominus(r)} \left( \int_{B_{r}} v^{2\theta}(t,x) \d
x\right)^{1/\theta} 
\left(  \int_{B_{r}} v^2(t,x) \d x \right)^{1/\theta'} \d t \nonumber \\
&\leq S\ \sup_{t \in I_\ominus(r)} \left( \int_{B_{r}} v^2(t,x) \d x
\right)^{1/\theta'} \times \\ &\qquad \times \left[\akon
\int_{Q_{_\ominus}(r)} \int_{B_{r}}
\frac{\bet{v(s,x)-v(s,y)}^2}{\bet{x-y}^{d+\alpha}} \d x \d y \d s + r^{-\alpha}
\int_{Q_\ominus(r)} v^2(s,x) \d x \d s \right], 
\end{align*}
where $S=S(d,\alpha_0)$.
Using \eqref{eq:cacciopoli} twice, $r \geq \frac12$ and \eqref{eq:K-2} yields
\begin{multline*}
\int_{Q_\ominus(r)} v^{2\kappa}(t,x) \d x \d t \leq c_4(d,\Lambda,\alpha_0) \left[ (q-1) \vartheta(q) \left( (R-r)^{-\alpha} + (R^{\alpha}-r^\alpha)^{-1} \right) \right]^{1/\theta'}
\times \\ 
\times \left[ (q-1) \vartheta(q) \left( (R-r)^{-\alpha} + (R^{\alpha}-r^\alpha)^{-1} \right) +1 \right]
\left[\int_{Q_{\ominus}(R)} v^2(s,x) \d x \d s \right]^{1+1/\theta'}.
\end{multline*}
Finally, we can estimate the coefficient by
\begin{align*}
c_4(d,\Lambda,\alpha_0) &\left[ (q-1) \vartheta(q) \left( (R-r)^{-\alpha} + (R^{\alpha}-r^\alpha)^{-1} \right) \right]^{\frac{1}{\theta'}}\\
&\qquad \qquad + \left[ (q-1) \vartheta(q) \left( (R-r)^{-\alpha} + (R^{\alpha}-r^\alpha)^{-1} \right)
\right]^{\kappa} \\
&\leq c_5(d,\Lambda,\alpha_0) \left[ q \vartheta(q) \left( (R-r)^{-\alpha} + (R^{\alpha}-r^\alpha)^{-1} \right)
\right]^\kappa \\
&\leq c_6(d,\Lambda,\alpha_0) q^{2\kappa} \left[ (R-r)^{-\alpha} + (R^\alpha-r^\alpha)^{-1}
\right]^\kappa\ ,
\end{align*}
which finishes the proof of \eqref{eq:moment-neg} by taking $p=q-1$ and
resubstituting $v=u^{\frac{1-q}{2}}$.
\end{proof}


\subsection{An estimate for \texorpdfstring{$\inf u$}{inf u}}
We apply the fact that the $p$-th moments of $u$ converge to the infimum for $p
\to -\infty$ and establish a local estimate on the infimum of $u$.

\begin{sa}[Moser iteration I] \label{thm:inf-est} Let $\frac12\leq r<R\leq
1$ and $0<p\leq1$. Then there is a constant $C=C(d, \alpha_0, \Lambda)>0$
such that for every nonnegative supersolution $u$ in $Q = I \times \Omega$, $Q \Supset Q_\ominus(R)$, with $u \geq
\epsilon >0$ in $Q$ the following estimate holds:
\begin{align}
\sup_{Q_{\ominus}(r)} \widetilde u^{-1} &\leq
\left(\frac{C}{G_1(r,R)}\right)^{1/p} \left( \int_{Q_{\ominus}(R)} \widetilde
u^{-p}(t,x) \d x \d t \right)^{1/p}, \label{eq:inf-est}
\end{align}
where $\widetilde u = u + \norm{f}_{L^\infty(Q)}$ and 
$G_1(r,R)= \begin{cases}
(R-r)^{d+\alpha}\quad &\text{if } \alpha \geq 1,\\
(R^\alpha-r^\alpha)^{(d+\alpha)/\alpha} &\text{if } \alpha <1.
\end{cases}$
\end{sa}
Note that on can estimate $G_1(r,R)$ independently of $\alpha$ as follows:
\[G_1(r,R) \geq (R-r)^{d+2} \wedge \left(\alpha_0
(R-r)\right)^{(d+2)/\alpha_0}.\]
The proof of \autoref{thm:inf-est} uses the method established in \cite[§
4]{moser-neu}.
\begin{proof} To shorten notation we define
\begin{equation}
 \cM_\ominus(u^{-1};r,p) = \cM_\ominus(r,p) = \left( \int_{Q_\ominus(r)}
\bet{u(t,x)}^{-p} \d x \d t \right)^{1/p}.
\end{equation}
Then the moment inequality \eqref{eq:moment-neg} reads
\[ \cM_\ominus(r, \kappa p) \leq A^{1/p}\ \
\cM_\ominus(R,p) \qquad \text{ with } A \text{ as in \eqref{eq:const-neg}}.\]
The strategy is to iterate \autoref{prop:step-neg} using a decreasing
sequence $r_0=R>r_1>r_2>\ldots>r$ of radii and the sequence of negative
exponents $p_m=-p \kappa^m$, $m \in \N_0$. 

Applying \eqref{eq:moment-neg} repeatedly we obtain the inequality
\begin{align*}
\cM^p_\ominus(r,p_{m+1}) \leq \cM^p_\ominus(r_{m+1},p_{m+1}) \leq
A_m^{1/\kappa^{m}} \cM^p_\ominus(r_m, p_m) \leq \cM^p_\ominus(r_0,p)
\prod_{j=0}^m A_j^{1/\kappa^{j}} .
\end{align*}
with $A_j = C(p_j+1)^2 \left( \left(r_j-r_{j+1}\right)^{-\alpha}+
(r_j^{\alpha}-r_{j+1}^\alpha)^{-1} \right)$.
Taking the limit $m \to \infty$, we find
\[ \sup_{Q_\ominus(r)} \widetilde u^{-p} = \lim_{m\to\infty}
\cM^p_\ominus(r,p_m)  
\leq \cM^p_\ominus(R,p) \prod_{j=0}^\infty A_j^{1/\kappa^{j}}  .\]
Consequently it now suffices to prove that
\begin{align} \label{eq:product}
 \prod\limits_{j=0}^\infty A_j^{1/\kappa^{j}} \leq
\frac{C}{G_1(r,R)}
\end{align}
for some suitable constant $C>0$. To establish \eqref{eq:product} we will choose
two different sequences $(r_m)$ in the cases $\alpha \in [1,2)$ and $\alpha\in
(\alpha_0,1)$:
\begin{itemize}
 \item $\alpha \in [1,2):$ Set $r_m = r+ \frac{(R-r)}{2^m}$. We can estimate
 \[ A_j \leq  c_1 (p\kappa^j + 1)^2 \left(r_j-r_{j+1}\right)^{-\alpha} \leq c_1 (2
\kappa^j)^2 \left(\frac{2^{j+1}}{R-r}\right)^\alpha \leq
\frac{c_2^j}{(R-r)^\alpha}. \]
Now the convergence of the infinite product in \eqref{eq:product} follows
immediately since\\ $\sum\limits_{j=0}^\infty \frac{j}{\kappa^{j}}\leq
\sum\limits_{j=0}^\infty \frac{j}{(1+\alpha_0/d)^{j}} \leq  c_3(\alpha_0,d)
<\infty$ and
\[ \prod\limits_{j=0}^\infty \left[ (R-r)^{-\alpha}\right]^{1/\kappa^{j}} =
(R-r)^{-\alpha \sum\limits_{j=0}^\infty \kappa^{-j}}  =(R-r)^{-\alpha
\frac{\kappa}{\kappa-1}}= \frac{1}{(R-r)^{d+\alpha}}\ .\]
This proves \eqref{eq:inf-est} in the case $\alpha \geq 1$.
\item $\alpha \in (\alpha_0,1]:$ Set $r_m = \left( r^{\alpha} +
\frac{(R^\alpha-r^\alpha)}{2^m} \right)^{1/\alpha}$. Then we obtain in a very similar way
as above
\[ \displaystyle A_j \leq \frac{c_4^j}{(R^\alpha-r^\alpha)} \quad \text{and}
\quad \prod\limits_{j=0}^\infty \left[
(R^\alpha-r^\alpha)^{-1}\right]^{1/\kappa^{j}} =
\frac{1}{(R^\alpha-r^\alpha)^{(d+\alpha)/\alpha}}.\]
This proves \eqref{eq:inf-est} in the case $\alpha <1.$
\end{itemize}
The lower bound on $G_1$ follows from the elementary inequalities
\begin{align*}
 (R-r)^{d+\alpha} &\geq (R-r)^{d+2},\\
 (R^\alpha - r^\alpha) & \geq \alpha\, 1^{\alpha-1} (R-r) \geq  \alpha_0 (R-r) \quad \text{if } \alpha_0 < \alpha < 1.
\end{align*}
The proof of \autoref{thm:inf-est} is complete.
\end{proof}


\subsection{Basic step for positive exponents}
The technique for the proof of the basic step for small positive exponents is very similar to the one used in the case of negative exponents. We state it separately and indicate the modifications
which are necessary to obtain the corresponding moment inequality.
\begin{prop}\label{prop:step-pos} Let $\frac12 \leq r <R \leq 1$ and $p \in (0,\kappa^{-1}]$ with $\kappa=1+ \frac{\alpha}{d}$. Then every nonnegative
supersolution $u$ in $Q= I\times \Omega$, $Q \Supset Q_\oplus(R)$, satisfies
the following inequality
\begin{align} 
  \left( \int_{Q_\oplus(r)} \widetilde
u^{\kappa p}(t,x) \d x \d t \right)^{1/\kappa} &\leq A'  \int_{Q_\oplus(R)}
\widetilde u^{p}(t,x) \d x \d t\ ,\label{eq:moment-pos}
\end{align}
where $\widetilde u = u + \norm{f}_{L^\infty(Q)}$ and $A'$ can be chosen as
\begin{equation} \label{eq:const-pos}
 A' = C' \left( \left(R-r\right)^{-\alpha}+ (R^{\alpha}-r^\alpha)^{-1} \right)
\qquad \text{with } C'=C'( d,\alpha_0,\Lambda).
\end{equation}
\end{prop}

\begin{proof} Let $u$ be a supersolution in $Q$ with $u \geq 0$ on $I \times
\R^d$. We set $\widetilde u =
u+\norm{f}_{L^\infty(Q)}$. If $f=0$ a.e. in $Q$ we set $\widetilde u = u+
\epsilon$ and pass to the limit $\epsilon \searrow 0$ in the end.

Set $q=1-p \in [1-\kappa^{-1},1)$ and define
\[ v(t,x) = \widetilde u^{\frac{1-q}{2}}(t,x), \qquad \phi(t,x) = \widetilde
u^{-q}(t,x) \psi^{2}(x)\]
with $\psi$ as in the proof of \autoref{prop:step-neg}. From \eqref{eq:steklov-solution} we obtain
\begin{multline}\label{eq:3-21} \int_{B_R} -\psi^{2}(x) \widetilde u^{-q}(t,x) \pt
\widetilde u(t,x) \d x\\ +\iint\limits_{\R^d\, \R^d} \left[ \widetilde
u(t,x)-\widetilde u(t,y) \right] \left[ \psi^{2}(y) \widetilde u^{-q}(t,y) -
\psi^{2}(x)\widetilde u^{-q}(t,x)\right] k_t(x,y) \d x \d y\\ \leq \int_{B_R} -
\psi^{2}(x) \widetilde u^{-q}(t,x) f(t,x) \d x.
\end{multline}
First we observe that for every small $h >0$
\begin{equation}\label{eq:lokalisierung}
\begin{split}
\iint\limits_{\R^d\, \R^d}&  \left[ \widetilde
u(t,x)-\widetilde u(t,y) \right] \left[ \psi^{2}(y) \widetilde u^{-q}(t,y) -
\psi^{2}(x)\widetilde u^{-q}(t,x)\right] k_t(x,y) \d x \d y \\
&= \iint\limits_{B_{R+h}\, B_{R+h}} \left[ \widetilde
u(t,x)-\widetilde u(t,y) \right] \left[ \psi^{2}(y) \widetilde u^{-q}(t,y) -
\psi^{2}(x)\widetilde u^{-q}(t,x)\right] k_t(x,y) \d x \d y \\
&\quad + 2 \int_{B_R} \int_{B_{R+h}^c} \left[ \widetilde
u(t,x)-\widetilde u(t,y) \right] \left[ -\psi^{2}(x)\widetilde u^{-q}(t,x)\right] k_t(x,y) \d y \d x\ .
\end{split}
\end{equation}
Using \eqref{eq:K-1}, the positivity of $\widetilde u$ and the fact that $\frac{(\psi(x)-\psi(y))^2}{\bet{x-y}^2} \leq (R-r)^{-2}$ we can estimate as follows:
\begin{align*}
 \int_{B_R} & \int_{B_{R+h}^c} \left[ \widetilde
u(t,x)-\widetilde u(t,y) \right] \left[ -\psi^{2}(x)\widetilde u^{-q}(t,x)\right] k_t(x,y) \d y \d x \\
&\geq -\int_{B_R} \widetilde u^{1-q}(t,x)  \left[ (R-r)^{-2}\int_{\bet{y-x} \leq R-r} \bet{x-y}^2 k_t(x,y) \d y + \int_{\bet{y-x} > R-r} k_t(x,y) \d y \right] \d x \\
&\geq -\Lambda (R-r)^{-\alpha} \int_{B_R} v^2(t,x) \d x.
\end{align*}
If $h \to 0$, this shows also that the decomposition in \eqref{eq:lokalisierung} is valid with $h=0$. Rewriting $\pt{v^2} =
(1-q)  \widetilde u^{-q} \pt{\widetilde u}$ and using that $\norm{f/\widetilde u}_{L^{\infty}(Q)} \leq 1$ we deduce from \eqref{eq:3-21}
\begin{multline*} \frac{-1}{1-q} \int_{B_R} \psi^{2}(x) \pt
v^2(t,x) \d x+\\ + \iint\limits_{B_R\, B_R} \left[ \widetilde
u(t,x)-\widetilde u(t,y) \right] \left[ \psi^{2}(y) \widetilde u^{-q}(t,y) -
\psi^{2}(x)\widetilde u^{-q}(t,x)\right] k_t(x,y) \d x \d y\\ 
\leq c_1 \Lambda (R-r)^{-\alpha} \int_{B_R}  v^2(t,x) \d x.
\end{multline*}
Applying \autoref{lem:guelle}(ii) we arrive at
\begin{multline*}
 \frac{-1}{1-q} \int_{B_R} \psi^{2}(x) \pt
v^2(t,x) \d x + \zeta_1(q) \iint\limits_{B_R\, B_R} [\psi(x) v(t,x) - \psi(y) v(t,y)]^2 k_t(x,y) \d x \d y \\
\leq c_1\Lambda (R-r)^{-\alpha} \int_{B_R}  v^2(t,x) \d x + \zeta_2(q) \iint\limits_{B_R\, B_R} (\psi(x)-\psi(y))^2 (v^2(t,x) + v^2(t,y)) k_t(x,y) \d x \d y\\
\end{multline*}
By the properties of $\psi$ and \eqref{eq:K-1} this implies
\begin{align*}
 - \int_{B_R} \psi^{2}(x) \pt
v^2(t,x) \d x &+ (1-q)\zeta_1(q) \iint\limits_{B_r\, B_r} [ v(t,x) -v(t,y)]^2 k_t(x,y) \d x \d y \\
&\leq c_2(d,\Lambda)(1-q) (R-r)^{-\alpha} \bigl(1+\zeta_2(q)\bigr) \int_{B_R}  v^2(t,x) \d x.
\end{align*}
We multiply this inequality with $\chi^2$, where $\chi \colon \R \to
[0,1]$ is defined by
$\chi(t)= \left( \frac{R^\alpha-t}{R^\alpha-r^\alpha} \wedge 1 \right) \vee 0$.
We integrate the resulting inequality from some $t \in I_\oplus(r)$ to
$R^\alpha$ and apply the same technique that we used to obtain
\eqref{eq:cacciopoli} from \eqref{eq:raum-absch} in \autoref{prop:step-neg}. As
a result we get
\begin{multline*}
\sup_{t \in I_{\oplus}(r)} \int_{B_{r}} v^2(t,x) \d x +
(1-q)\zeta_1(q) \int\limits_{Q_{\oplus}(r)} \int\limits_{B_{r}} \left[ v(s,x) - v(s,y) \right]^2 k_s(x,y) \d
x \d y \d s \\
\leq c_3(d,\Lambda) \left[ (1-q)(1+ \zeta_2(q)) (R-r)^{-\alpha} + 
\left(R^{\alpha}-r^\alpha\right)^{-1} \right] \int\limits_{Q_{\oplus}(R)} v^2(s,x) \d
x \d s.
\end{multline*}
We estimate the coefficients by
\begin{align*}
 (1-q)\zeta_1(q)&=\frac{2q}{3} \geq \frac23 \frac{\alpha_0}{d+2} =:
c_4(d,\alpha_0) \\
 (1-q)(1+ \zeta_2(q)) &\leq 1 + (1-q) \zeta_2(q) \leq 1 + 4q + \frac{9(1-q)}{q}
\leq 5 + 9 \frac{d+2}{\alpha_0} =: c_5(d,\alpha_0),
\end{align*}
which implies
\begin{multline*}
\sup_{t \in I_{\oplus}(r)} \int_{B_{r}} v^2(t,x) \d x +
c_4 \int\limits_{Q_{\oplus}(r)} \int\limits_{B_{r}} \left[ v(s,x) - v(s,y) \right]^2 k_s(x,y) \d
x \d y \d s \\
\leq c_6(d,\Lambda,\alpha_0) \left[ (R-r)^{-\alpha} + \left(R^{\alpha}-r^\alpha\right)^{-1} \right] \int\limits_{Q_{\oplus}(R)} v^2(s,x) \d
x \d s.
\end{multline*}
Applying Sobolev's inequality as in the proof of \autoref{prop:step-neg}  we obtain
\begin{multline*}
\int\limits_{Q_\oplus(r)} v^{2\kappa}(t,x) \d x \d t \leq c_7(d,\Lambda,\alpha_0) \left[ (R-r)^{-\alpha} + \left(R^{\alpha}-r^\alpha\right)^{-1} \right]^{1/\theta'}
\times \\ 
\times \left[ (R-r)^{-\alpha} + \left(R^{\alpha}-r^\alpha\right)^{-1}  +1\right]
\left[\int_{Q_{\oplus}(R)} v^2(s,x) \d x \d s \right]^{1+1/\theta'}.
\end{multline*}
We can estimate the coefficient by
\begin{multline*}
 \left[ (R-r)^{-\alpha} + \left(R^{\alpha}-r^\alpha\right)^{-1} \right]^{1/\theta'} +
 \left[ (R-r)^{-\alpha} + \left(R^{\alpha}-r^\alpha\right)^{-1} \right]^{\kappa} \\
\leq c_8(d,\Lambda,\alpha_0) \left[ (R-r)^{-\alpha} + (R^\alpha-r^\alpha)^{-1} \right]^\kappa\ .
\end{multline*}
This finishes the proof of \autoref{prop:step-pos} by resubstituting $q=1-p$ and $v=\widetilde u^{\frac{1-p}{2}}$.
\end{proof}
Note that $\kappa^{-1}$, the upper bound on $p$, can be replaced by any number less than $1$.

\subsection{An estimate for small positive moments of \texorpdfstring{$u$}{u}}

The aim of this subsection is to estimate the $L^1$-norm of supersolutions $u$
from above by the $L^1$-norm of $u^p$ for small values of $p>0$. 

\begin{sa}[Moser iteration II]
\label{thm:p-est} Let $\frac12\leq r<R\leq 1$ and $p \in (0,\kappa^{-1})$ with $\kappa=1+ \frac{\alpha}{d}$. Then there are constants $C,
\omega_1, \omega_2>0$ depending only on $d, \alpha_0, \Lambda$, such that for every nonnegative supersolution $u$ in $Q= I\times \Omega$, $Q \Supset Q_\oplus(R)$, the following estimate holds:
\begin{align}
\int_{Q_\oplus(r)} \widetilde u(t,x) \d x \d t &\leq \left( \frac{C}{\bet{Q_\oplus(1)} G_2(r,R)}
\right)^{1/p-1} \left( \int_{Q_{\oplus}(R)} \widetilde u^{p}(t,x) \d x \d t
\right)^{1/p}, \label{eq:p-est}
\end{align}
where $\widetilde u = u + \norm{f}_{L^\infty(Q(R))}$ and $G_2(r,R) =
(R-r)^{\omega_1} \wedge \left( \alpha_0 (R-r) \right)^{\omega_2}$.
\end{sa}

\begin{proof} We adopt the proof of \cite[Lemma 2.2]{zacher}. Without loss of
generality we may assume $\bet{Q_\oplus(1)}=1$; otherwise replace the measure
$\d x \otimes \d t$ in \eqref{eq:p-est} by $\bet{Q_\oplus(1)}^{-1}(\d x \otimes
\d t)$. 

Assume that $\alpha\geq 1$. Similar to the proof of \autoref{thm:inf-est} define
for $j=1,\ldots,n$
\[ p_j = \kappa^{-j} \qquad \text{and} \qquad r_j = r + \frac{R-r}{2^j}.\]
Furthermore, setting
$ \cM_\oplus(u;r,p) = \cM_\oplus(r,p) = \left( \int_{Q_\oplus(r)}
\bet{u(t,x)}^{p} \d x \d t \right)^{1/p}$,
the assertion \eqref{eq:moment-pos} of \autoref{prop:step-pos} reads
\[ \cM_\oplus(r,\kappa p) \leq A'^{1/p}
\cM_\oplus(R,p) \qquad \text{with } A' \text{ as in \eqref{eq:const-pos}}.\]

Iterating \eqref{eq:moment-pos} $n$ times, $n \in \N$, with $p_j$ and $r_j$ as
above we obtain
\begin{align*}
\cM_\oplus(r,1) &\leq \cM_\oplus(r_n,1)  = \cM_\oplus(r_n,p_1 \kappa) \leq
\left(\frac{C'\ 2^{\alpha n}}{(R-r)^{\alpha}}\right)^\kappa
\cM_\oplus(r_{n-1},p_1)\\
&\leq \left(\frac{C'\ 2^{\alpha n}  }{(R-r)^{\alpha}}\right)^\kappa \left(
\frac{C'\ 2^{\alpha(n-1)}}{(R-r)^{\alpha}}\right)^{\kappa^2}
\cM_\oplus(r_{n-2},p_2)
\leq \cM_\oplus(r_0,p_n) \prod_{j=1}^n \left(\frac{C'\
2^{\alpha(n-j+1)}}{(R-r)^{\alpha}}\right)^{\kappa^j}
\end{align*}
Employing the formulae
\begin{align*}
 \sum_{j=1}^n \kappa^j &= \frac{\kappa}{\kappa-1}\left(\frac{1}{p_n} - 1\right)
= \frac{d+\alpha}{\alpha} \left(\frac{1}{p_n} - 1\right) ,& \quad  \quad
\sum_{j=1}^n (n-j+1) \kappa^j &\leq \frac{\kappa^3}{(\kappa-1)^3} \left(
\frac{1}{p_n} -1\right),
\end{align*} 
we deduce
\begin{align}
 \cM_\oplus(r,1) \leq \left[  \frac{ 2^{\frac{\alpha \kappa^3}{(\kappa-1)^3}}
C'^{\frac{\kappa}{\kappa-1}}}{(R-r)^{ \frac{\alpha\kappa}{\kappa-1}}}
\right]^{\frac{1}{p_n}-1} \cM_\oplus(r_0,p_n). \label{eq:Adam}
\end{align}
Now, for $p \in (0,\kappa^{-1})$ fix $n \in \N$ such that $p_n \leq p <
p_{n-1}$. Thus
\begin{align}
 \frac{1}{p_n} - 1 &= \kappa^n-1 \leq \kappa^n + \kappa^{n-1} - \kappa -1 =
(1+\kappa)(\kappa^{n-1}-1) = (1+\kappa)\left( \frac{1}{p_{n-1}-1} \right)
\nonumber \\
 &\leq (1+\kappa) \left( \frac{1}{p}-1\right). \label{eq:Bernd}
\end{align}
Additionally we have by Hölder's inequality
\[ \cM_\oplus(r_0,p_n) = \cM_\oplus(R,p_n) \leq \bet{
Q_\oplus(R)}^{\frac{1}{p_n}-\frac{1}{p}} \cM_\oplus(R,p) = \cM_\oplus(R,p).\]  
Combining \eqref{eq:Adam}, \eqref{eq:Bernd} and the latter inequality we obtain
\begin{align*}
 \cM_\oplus(r,1) \leq \left[ 
\frac{ 2^{\frac{\alpha \kappa^3}{(\kappa-1)^3}}
C'^{\frac{\kappa}{\kappa-1}}}{(R-r)^{(d+\alpha)} }
\right]^{(1+\kappa)\left(\frac{1}{p}-1\right)} \left(
 \int_{Q_\oplus(1)} u^p \right)^{1/p}.
\end{align*}
Finally, since $\alpha \in (\alpha_0,2)$ and $\kappa \in (1+\alpha_0/d,1+2/d)$,
there is a constant $c=c(d,\alpha_0,C')$ such that
\begin{equation} \label{eq:riesenprodukt}
 \cM_\oplus(r,1) \leq \left[ \frac{c}{
(R-r)^{(1+\kappa)(d+\alpha)}} \right]^{\frac{1}{p}-1} \cM_\oplus(R,p).
\end{equation}
This proves \eqref{eq:p-est} in the case $\alpha\geq 1$ with $\omega_1 = 2 d + 6 + 4/d \geq (1+\kappa)(d+\alpha)$.

For $\alpha<1$ apply the same
arguments to the sequence of radii defined by $\widetilde r_j=\left( r^{\alpha} +
\frac{R^{\alpha}-r^\alpha}{2^j}\right)^{1/\alpha}$. For these radii the inequality corresponding to \eqref{eq:riesenprodukt} reads
\[ \cM_\oplus(r,1) \leq \left[ \frac{c}{
(R^\alpha-r^\alpha)^{(1+\kappa)\frac{d+\alpha}{\alpha}}} \right]^{\frac{1}{p}-1} \cM_\oplus(R,p).\]
This proves \eqref{eq:p-est} in the case $\alpha \in (\alpha_0,1)$ with $\omega_2=2d/\alpha_0 + 3 + 2/d \geq (1+\kappa)\frac{d+\alpha}{\alpha}$. 

As we can see from the proof, if $\alpha_0 \geq 1$, one may choose
$G_2(r,R)= (R-r)^{\omega_1}$. \autoref{thm:p-est} is proved.
\end{proof}
%
\section{An estimate for \texorpdfstring{$\log u$}{log u}}
\label{sec:estimates_log_u}
%
The following lemma provides a lower bound for the nonlocal term in \eqref{eq:def-solution} when applying $u^{-1}$ times some cut-off function as test function. See \cite[Proposition 4.9]{BBCK09} for a similar result.
\begin{lem}\label{lem:6} Let $I \subset \R$ and $\psi \colon \R^d \to [0,\infty)$ be a continuous function satisfying $\Supp[\psi] = \overline{B_R}$ for some $R>0$ and $ \sup_{t \in I} \cE_t(\psi,\psi) < \infty$. Then the following computation rule holds for $w \colon I \times \R^d \to [0,\infty)$:
\begin{align*}
\cE_t(w, -\psi^2 w^{-1}) \geq \iint\limits_{B_R\, B_R} \psi(x) \psi(y) \left( \log \frac{w(t,y)}{\psi(y)} - \log \frac{w(t,x)}{\psi(x)} \right)^2 k_t(x,y) \d x \d y - 3\ \cE_t(\psi,\psi)
\end{align*}
\end{lem}
\begin{bem}
We apply the rule above only in cases where all terms are finite. 
\end{bem}

\begin{proof} Fix $t\in I$. First of all we note that
\begin{align}
\cE_t(w, &-\psi^2 w^{-1}) = \iint\limits_{\R^d\, \R^d} \left[ w(t,y)-w(t,x) \right] \left[ \psi^2(x) w^{-1}(t,x) - \psi^2(y) w^{-1}(t,y)\right] k_t(x,y) \d y \d x \nonumber \\
&\geq \iint\limits_{B_R\, B_R} \psi(x)\psi(y) \left[ \frac{\psi(x) w(t,y)}{\psi(y) w(t,x)} + \frac{\psi(y) w(t,x)}{\psi(x) w(t,y)} - \frac{\psi(y)}{\psi(x)} - \frac{\psi(x)}{\psi(y)}\right] k_t(x,y) \d y \d x \nonumber \\
& \qquad +2 \iint\limits_{B_R\, B_R^c}  \left[ w(t,y)-w(t,x) \right] \left[ \psi^2(x) w^{-1}(t,x) - \psi^2(y) w^{-1}(t,y)\right] k_t(x,y) \d y \d x  \label{eq:4-1}\\
& \qquad +\iint\limits_{B_R^c\, B_R^c} \left[ w(t,y)-w(t,x) \right] \left[ \psi^2(x) w^{-1}(t,x) - \psi^2(y) w^{-1}(t,y)\right] k_t(x,y) \d y \d x \nonumber
\end{align}
Because of $\Supp[\psi] = \overline{B_R}$ the third term on the right-hand side vanishes.

To estimate the first term on the right-hand side we apply the inequality
\[ \frac{(a-b)^2}{ab} = (a-b) \left(b^{-1}-a^{-1}\right) \geq \left( \log a - \log b \right)^2\qquad \text{for } a,b >0 \]
to $a=A_t(x,y)= \frac{w(t,y)}{w(t,x)}$ and $b=B(x,y)=\frac{\psi(y)}{\psi(x)}$, $x,y \in B_{R}$:\enlargethispage{3em}
\begin{align*}
 \frac{\psi(x) w(t,y)}{\psi(y) w(t,x)} &+ \frac{\psi(y) w(t,x)}{\psi(x) w(t,y)} - \frac{\psi(y)}{\psi(x)} - \frac{\psi(x)}{\psi(y)} \\
 &=   \frac{A_t(x,y)}{B(x,y)} + \frac{B(x,y)}{A_t(x,y)} - 2 - \left( \sqrt{B(x,y)}- \frac{1}{\sqrt{B(x,y)}} \right)^2 \\
 &\geq \left( \log \frac{w(t,y)}{\psi(y)} - \log \frac{w(t,x)}{\psi(x)} \right)^2 - \left( \frac{\psi(x)}{\psi(y)} + \frac{\psi(y)}{\psi(x)} - 2 \right).
\end{align*}
Hence,
\begin{multline} \label{eq:4-3}
 \iint\limits_{B_R\, B_R} \psi(x)\psi(y) \left[ \frac{\psi(x) w(t,y)}{\psi(y) w(t,x)} + \frac{\psi(y) w(t,x)}{\psi(x) w(t,y)} - \frac{\psi(y)}{\psi(x)} - \frac{\psi(x)}{\psi(y)}\right] k_t(x,y) \d y \d x \\
 \geq \iint\limits_{B_R\, B_R} \psi(x)\psi(y) \left( \log \frac{w(t,y)}{\psi(y)} - \log \frac{w(t,x)}{\psi(x)} \right)^2 \d y \d x - \cE_t(\psi,\psi). 
\end{multline}
Finally, we estimate the second term using the nonnegativity of $w(t,\cdot)$ in $\R^d$:
\begin{align}
 \iint\limits_{B_R\, B_R^c} &\left[ w(t,y)-w(t,x) \right] \left[ \psi^2(x) w^{-1}(t,x) - \psi^2(y) w^{-1}(t,y)\right] k_t(x,y) \d y \d x  \nonumber \\
&= \iint\limits_{B_R\, B_R^c} \left[ w(t,y)-w(t,x) \right] \left[ \psi^2(x) w^{-1}(t,x) \right] k_t(x,y) \d y \d x \nonumber \\
&= \int_{B_R} \frac{\psi^2(x)}{w(t,x)} \int_{B_R^c} w(t,y) k_t(x,y) \d y \d x - \int_{B_R} \psi^2(x) \int_{B_R^c} k_t(x,y) \d y \d x \nonumber \\
&\geq - \int_{B_R} \int_{B_R^c} \left[ \psi(x)-\psi(y)\right]^2 k_t(x,y) \d y \d x \geq - \cE_t(\psi,\psi). \label{eq:4-4}
\end{align}
Applying the estimates \eqref{eq:4-3} and \eqref{eq:4-4} in \eqref{eq:4-1} finishes the proof of \autoref{lem:6}.
\end{proof}
\begin{prop} \label{prop:loglemma} Assume $k \in \cK(\alpha_0,\Lambda)$ for some $\alpha_0 \in (0,2)$ and $\Lambda \geq 1$. Then there is $C=C(d,\alpha_0,\Lambda)>0$ such that for every supersolution $u$ of \eqref{eq:par_equation} in $Q = (-1,1)\times B_2(0)$ which satisfies $u \geq \epsilon >0$ in $(-1,1) \times \R^d$, there is a constant $a=a(\widetilde u) \in \R$ such that the following inequalities hold simultaneously:
\begin{subequations}
\begin{align} 
\forall s >0 \colon (\d t \otimes \d x)\left( Q_{\oplus}(1) \cap \left\{ \log \widetilde u < -s-a \right\} \right) \leq \frac{C \bet{B_1} }{s}, \label{eq:log-neg} \\
\forall s >0 \colon (\d t \otimes \d x)\left( Q_{\ominus}(1) \cap \left\{ \log \widetilde u > s-a \right\} \right) \leq \frac{C \bet{B_1}}{s}, \label{eq:log-pos}
\end{align}
\end{subequations}
where $ \widetilde u = u + \norm{f}_{L^\infty(Q)}$.
\end{prop}
\begin{proof} In the course of the proof we introduce constants $c_1,c_2,c_3,c_4$ that may depend on $d$, $\alpha_0$, and $\Lambda$.
We use the test function $\phi(t,x)=\psi^2(x) \widetilde u^{-1}(t,x)$ in \eqref{eq:steklov-solution}, where 
\[\psi^2(x)=\left[ \left(\tfrac32-\bet{x}\right) \wedge 1\right] \vee 0,\qquad x \in \R^d,\]
and we write $v(t,x)=-\log \frac{\widetilde u(t,x)}{\psi(x)}$. Thus we have for a.e. $t \in (-1,1)$
\begin{align*}
\int_{B_{3/2}} \psi^2(x) \pt v(t,x) \d x + \cE_t(\widetilde u,-\psi^2 \widetilde u^{-1}) \leq - \int_{B_{3/2}} \psi^2(x) \widetilde u^{-1}(t,x) f(t,x) \d x.
\end{align*}
Note that $\cE_t(\widetilde u,-\psi^2 \widetilde u^{-1})$ is finite since
$\widetilde u(t,\cdot) \in H^{\alpha/2}_{loc}(B_2)$ for a.e. $t \in (-1,1)$ and
$\supp \psi = \overline{B_{3/2}}$.
Applying \autoref{lem:6} and $\norm{f/\widetilde u}_{L^{\infty}(Q)} \leq 1$ we obtain
\begin{multline*}
\int_{B_{3/2}} \psi^2(x) \pt v(t,x) \d x\ +\\
+ \iint\limits_{B_{3/2}\, B_{3/2}} \psi(x) \psi(y) \left[ v(t,y) - v(t,x) \right]^2 k_t(x,y) \d x \d y
\leq \bet{B_{3/2}} + 3 \cE_t(\psi,\psi).
\end{multline*}
Now we apply the weighted Poincaré-inequality, \autoref{lem:poincare}, to the second term and use the fact $\sup\limits_{t \in(-1,1)} \cE_t(\psi,\psi)\leq C$ for some constant $C=C(d,\alpha_0,\Lambda)$. We obtain
\[
 \int_{B_{3/2}} \psi^2(x) \pt v(t,x) \d x + c_1 \int_{B_{3/2}} \left[ v(t,x) - V(t) \right]^2 \psi^2(x) \d x
 \leq c_2 \bet{B_1}, \]
where \[ V(t) = \frac{\int_{B_{3/2}} v(t,x) \psi^2(x)\d x}{\int_{B_{3/2}} \psi^2(x)\d x}.\]
The proof now proceeds as in the case of local operators. Our presentation uses ideas from \cite[pp. 120-123]{moser-alt} and \cite[Lemma 5.4.1]{saloff}.

Integrating the above inequality over $[t_1,t_2] \subset (-1,1)$ yields
\begin{equation} \label{eq:nonlocal-finished}
 \left[ \int_{B_{3/2}} \psi^2(x) v(t,x) \d x \right]_{t=t_1}^{t_2} + c_1 \int_{t_1}^{t_2} \int_{B_{3/2}} \left[ v(t,x) - V(t) \right]^2 \psi^2 \d x \leq c_2(t_2-t_1) \bet{B_1}\ .
\end{equation}
Dividing by $\int_{B_{3/2}} \psi^2$, using $\int_{B_{3/2}} \psi^2 \leq 2^d \bet{B_1}$ and $\psi=1$ in $B_1$, we obtain
\begin{align} 
 V(t_2)-V(t_1) + \frac{c_3^{-1}}{\bet{B_1} } \int_{t_1}^{t_2} \int_{B_1} \left[ v(t,x) - V(t) \right]^2 \d x \d t \leq c_2(t_2-t_1) \quad \text{with } c_3=\frac{2^d}{c_1}, \label{eq:V-diff}
 \intertext{or equivalently}
 \frac{V(t_2)-V(t_1)}{t_2-t_1} + \frac{c_3^{-1}}{\bet{B_1} (t_2-t_1)} \int_{t_1}^{t_2} \int_{B_1} \left[ v(t,x) - V(t) \right]^2 \d x \d t \leq c_2 
\end{align}
Assume that $V(t)$ is differentiable. Taking the limit $t_2  \to t_1$ the above inequality yields
\begin{equation} \label{eq:abl-V}
V'(t) + \frac{c_3^{-1}}{\bet{B_1}} \int_{B_1} \left[ v(t,x) - V(t) \right]^2 \d x \leq c_2,\qquad \text{for a.e. } t \in (-1,1).
\end{equation}

Now set
\[ w(t,x) = v(t,x) - c_2 t,\qquad W(t)=V(t)-c_2 t, \]
such that \eqref{eq:abl-V} reads
\begin{equation} \label{eq:abl-W}
 W'(t) + \frac{c_3^{-1}}{\bet{B_1}} \int_{B_1} \left[ w(t,x) - W(t) \right]^2 \d x \leq 0 \text{ for a.e. } t \in (-1,1),\qquad W(0)=a,
\end{equation}
where $a$ is a constant depending on $u$. Note that by the latter inequality $W$ is nonincreasing in $(-1,1)$.

We work out here the details for the proof of \eqref{eq:log-neg}. It is straightforward to mimic the arguments for the proof of \eqref{eq:log-pos}. Define for $t \in (0,1)$ and $s>0$ the set 
\begin{equation} 
 L^\oplus_s(t)=\left\{ x \in B_1(0)\colon w(t,x) > s + a\right\}.
\end{equation}
Noting that $W(t)\leq a$ for a.e. $t \in (0,1)$, we obtain for such $t$ and $x \in L^\oplus_s(t)$
\[  w(t,x) - W(t) \geq s + a - W(t) > 0.\]
Using this in \eqref{eq:abl-W} yields
\[ W'(t) + \frac{c_3^{-1}}{\bet{B_1}} \bet{L^\oplus_s(t)} (s +a-W(t))^2 \leq 0,\]
which is equivalent to
\[ \frac{-c_3\, W'(t)}{(s+a-W(t))^2} \geq \frac{\bet{L^\oplus_s(t)}}{\bet{B_1}}\ .\]
Intergrating this inequality over $t \in (0,1)$ we obtain
\[ \frac{c_3}{s}  \geq \left[ \frac{c_3}{s+a-W(\tau)}\right]_{\tau=0}^1 \geq \frac{1}{\bet{B_1}} \int_{0}^1 \bet{L^\oplus_s(t)} \d t = \frac{\bet{  Q_{\oplus}(1) \cap \left\{ w > s + a \right\}}}{\bet{B_1}}\]
and replacing $w$ again by $w(t,x) = v(t,x) - c_2 t = -\log \widetilde u - c_2 t$ in $Q_\oplus(1)$ yields
\begin{equation} \label{eq:Walter}
 \bet{ Q_{\oplus}(1) \cap \left\{ \log  \widetilde u + c_2 t < -s - a \right\}} \leq  \frac{c_3 \bet{B_1}}{s}.
\end{equation}
Finally,
\begin{align*}
\bet{  Q_{\oplus}(1) \cap \left\{ \log \widetilde u < -s - a \right\}} &\leq \bet{ Q_{\oplus}(1) \cap \left\{ \log \widetilde u + c_2 t  < -s/2 - a \right\}} + \bet{  Q_{\oplus}(1) \cap \left\{  c_2 t  > s/2 \right\}} \\
&\leq \frac{2c_3}{s} \bet{B_1} + \left(1-\frac{s}{2c_2}\right) \bet{B_1} \leq \frac{c_4}{s}. 
\end{align*}

In case that $V$ is only continuous in $(-1,1)$ we derive the result in a different manner, cf. \cite[Lemma 6.21]{lieberman}:
For $\epsilon_0>0$ there is $\delta>0$ such that for $t_2 < t_1 + \delta$
\[ \bet{v(t,x) - V(t)}^2 \leq 2 \bet{v(t,x) - V(t_2)}^2 + 2 \bet{V(t_2)-V(t)}^2 \leq 2 \bet{v(t,x) - V(t_2)}^2 + 2 \epsilon_0^2.\]
Hence, by \eqref{eq:V-diff} we obtain for $t_2 < t_1 + \delta$
\[ V(t_2)-V(t_1) + \frac{c_3^{-1}}{\bet{B_1} } \int_{t_1}^{t_2} \int_{B_1} \left[ v(t,x) - V(t_2) \right]^2 \d x \d t \leq \left( 2 c_2+ 2 c_3^{-1} \epsilon_0^2\right) (t_2-t_1).  \]
Defining 
\[ w(t,x) = v(t,x) - \left( 2 c_2+ 2 c_3^{-1} \epsilon_0^2\right) t,\qquad W(t)=V(t)- \left( 2 c_2+ 2 c_3^{-1} \epsilon_0^2\right) t, \]
the latter inequality reads
\[ W(t_2) - W(t_1) + \frac{c_3^{-1}}{\bet{B_1} } \int_{t_1}^{t_2} \int_{B_1} \left[ w(t,x) - W(t_2) + (2 c_2+ 2 c_3^{-1} \epsilon_0^2) (t_2-t)  \right]^2 \d x \d t \leq 0. \]
Using the fact that for $t,t_2 \in (0,1)$ and $x \in L_s^\oplus(t)$ we have $w(t,x) - W(t_2) > s + a - W(t_2) \geq 0$, we can omit the term $(2 c_2+ 2 c_3^{-1} \epsilon_0^2) (t_2-t)$ in the integral and deduce that for $t_2 < t_1 + \delta$
\begin{align*}
 \frac{W(t_2)-W(t_1)}{(s+a - W(t_2))^2} + \frac{c_3^{-1}}{\bet{B_1} } \int_{t_1}^{t_2} \bet{L_s^\oplus(t)} \d t \leq 0.
\end{align*}
Again, since $W$ is nonincreasing, this implies
\begin{equation} 
\begin{split}
\label{eq:W-step}
 \frac{c_3^{-1}}{\bet{B_1} } \int_{t_1}^{t_2} \bet{L_s^\oplus(t)} \d t &\leq \frac{W(t_1)-W(t_2)}{(s+a - W(t_1))(s+a - W(t_2))}\\
 &= \frac{1}{s+a-W(t_1)} - \frac{1}{s+a-W(t_2)}\ .
\end{split}
\end{equation}
Choosing $k \in \N$ such that $\frac1k < \delta$, writing
\[ \int_{0}^1 \bet{L_s^\oplus(t)} \d t = \sum_{j=0}^{k-1} \int_{\frac{j}{k}}^{\frac{j+1}{k}} \bet{L_s^\oplus(t)} \d t,\]
and applying \eqref{eq:W-step} in each summand, we establish \eqref{eq:Walter}. Using the same arguments as above we establish \eqref{eq:log-neg}.
This finishes the proof of \autoref{prop:loglemma}.
\end{proof}
%
\section{Proof of the Harnack inequality}\label{sec:harnack}

The aim of this section is to prove \autoref{thm:harnack}. Our proof relies on
the well-known idea of Bombieri and Giusti. The following abstract lemma extends
the idea of \cite{bombieri} to the parabolic case. It was first proved in \cite[pp. 731-733]{moser-neu}. The version below can be found in \cite[Section 2.2.3]{saloff}.
\begin{lem} \label{prop:bombieri}
Let $(U(r))_{\theta\leq r\leq 1}$ be a nondecreasing family of domains $U(r)
\subset \R^{d+1}$ and let $m, c_0$ be positive constants, $\theta \in
[1/2,1]$, $\eta \in (0,1)$ and $0<p_0\leq \infty$. Furthermore assume that $w$
is a positive, measurable function defined on $U(1)$ which satisfies
\begin{align} \label{eq:ass1}
\Bigl( \int_{U(r)} w^{p_0} \Bigr)^{1/p_0} \leq \left( \frac{c_0}{(R-r)^m
\bet{U(1)}}\right)^{1/p-1/p_0} \Bigl( \int_{U(R)} w^p \Bigr)^{1/p} < \infty .
\end{align}
for all $r,R \in [\theta,1], r<R$ and for all $p \in (0,1 \wedge \eta p_0)$.

Additionally suppose that
\begin{equation} \label{eq:ass2}
\forall s>0 \colon\quad \bet{ U(1) \cap \{ \log w > s \}} \leq \frac{c_0}{s}
\bet{U(1)}.
\end{equation}
Then there is a constant $C=C(\theta,\eta,m,c_0, p_0)$ such that
\begin{equation} \label{eq:ergebnis}
 \Bigl( \int_{U(\theta)} w^{p_0} \Bigr)^{1/p_0} \leq C \bet{U(1)}^{1/p_0}\ .
\end{equation}
\end{lem}

\begin{proof}[Proof of \autoref{thm:harnack}]
Let $u$ as in the assumption and define $\widetilde u = u + \norm{f}_{L^\infty(Q)}$. If $f=0$ a.e. on $Q$ we set $\widetilde u = u+ \epsilon$ and pass to the limit $\epsilon \searrow 0$ in the end. 

Furthermore, set $w = e^{-a} \widetilde u^{-1}$ and $\widehat w = w^{-1} = e^{a} \widetilde u$,
where $a=a(\widetilde u)$ is chosen according to \autoref{prop:loglemma}, i.e. there is $c_1>0$ such that for every $s>0$
\begin{align} \label{eq:horst}
\bet{ Q_{\oplus}(1) \cap \{\log w > s \} } \leq \frac{c_1 \bet{B_1}}{s},\quad \text{and} \quad \bet{ Q_{\ominus}(1) \cap \{\log \widehat w > s \} } \leq \frac{c_1 \bet{B_1}}{s}.
\end{align}
The strategy of the proof is to apply \autoref{prop:bombieri} twice: on the one hand to $w$ and a family of domains  $\cU=(U(r))_{\theta \leq r \leq 1}$ -- and on the other hand to $\widehat w$ and a family of domains ${\widehat \cU}=(\widehat U(r))_{\widehat \theta \leq r \leq 1}$. We consider the case $\alpha \geq 1$ first and define the families $\cU$ and $\widehat \cU$ by
\begin{align*}
U(1) &= Q_{\oplus}(1),\quad &\theta &= \frac1{2},& U(r) &= \left(1- {r^\alpha}, 1\right) \times B_{r},\\
\widehat U(1) &= Q_{\ominus}(1), &\widehat \theta &=\frac12 & \widehat U(r)&=\left(-1, -1+r^\alpha \right) \times B_{r}
\end{align*}
\begin{figure}[ht]
\centering
\includegraphics{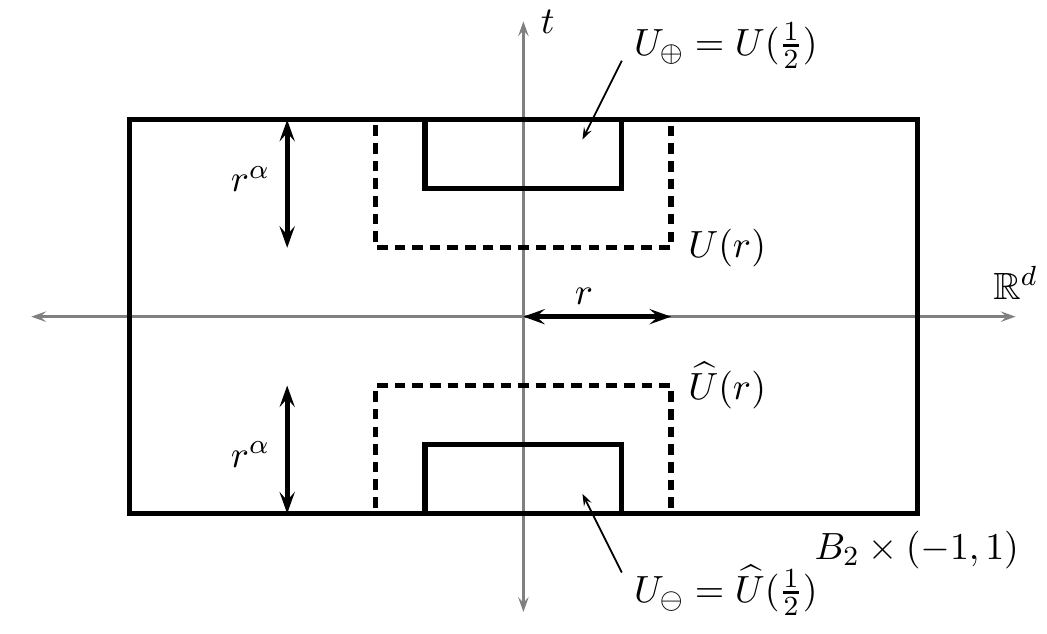}
\end{figure}
By virtue of \eqref{eq:horst} we see that condition \eqref{eq:ass2} is satisfied for both $w$ and $\widehat w$.

We apply \autoref{thm:inf-est} to $(w,\cU)$ with $p_0=\infty$ and arbitrary $\eta$. We also apply \autoref{thm:p-est} to $(\widehat w,\widehat \cU)$ with $\widehat p_0=1$ and $\widehat \eta=\frac{d}{d+2} \leq \kappa^{-1}$. In both cases the corresponding condition \eqref{eq:ass1} is satisfied. Note that the domains $U(r)$ and $\widehat U(r)$ are obtained from $Q_\ominus(r)$ and $Q_\oplus(r)$, respectively, by shiftings in time, i.e. transformations of the type $(t,x) \mapsto (t + \tau, x)$,
which do not affect neither \eqref{eq:inf-est} nor \eqref{eq:p-est}.

All in all, application of \autoref{prop:bombieri} yields
\begin{align*}
\sup_{U(\theta)} w = e^{-a}  \sup_{U(\theta)} \widetilde u^{-1} \leq C \qquad \text{ and } \quad \norm{\widehat w}_{L^1(\widehat U(\widehat \theta))} = e^{a} \norm{\widetilde u}_{L^1(\widehat U(\widehat \theta))} \leq \widehat C.
\end{align*}
Multiplying these two inequalities eliminates $a$ and yields
\[ \norm{\widetilde u}_{L^1(\widehat U(\widehat \theta))} \leq c_2 \inf_{U(\theta)} \widetilde u\]
for a constant $c_2= C\, \widehat C$ that depends only on $d,\alpha_0$ and $\Lambda$.
This proves \eqref{eq:harnack-speziell} in the case $\alpha \geq 1$ observing that $U_\oplus = U(\theta)$, $U_\ominus = \widehat U(\widehat \theta)$ and
\[ \norm{u}_{L^1(U_\ominus)} \leq \norm{\widetilde u}_{L^1(U_\ominus)} \leq c_2 \left( \inf_{U_{\oplus}} u + \norm{f}_{L^\infty(Q)} \right). \]
If $\alpha<1$, we define the domains $\cU$ and $\widehat \cU$ slightly differently, namely
\begin{align*}
U(1) &= Q_{\oplus}(1),\quad &\theta &= \left(\tfrac1{2}\right)^{\alpha},& U(r) &= \left(1- {r}, 1\right) \times B_{r^{1/\alpha}},\\
\widehat U(1) &= Q_{\ominus}(1), &\widehat \theta &=\left(\tfrac1{2}\right)^{\alpha}, & \widehat U(r)&=\left(-1, -1+r \right) \times B_{r^{1/\alpha}}.
\end{align*}
The same reasoning as above applies to these domains and hence \eqref{eq:harnack-speziell} is proved for all $\alpha \in (\alpha_0,2)$.
\end{proof}
The following corollary will be used to derive Hölder continuity in the next section.
\begin{cor} \label{cor:silvestre}
Let $\sigma \in (0,1)$ and $D_\ominus=(-2,-2+\left(\tfrac12\right)^\alpha) \times B_{1/2}$, $D_\oplus=(-\left(\tfrac12\right)^\alpha,0) \times B_{1/2}.$
There exist $\epsilon_0, \delta \in (0,1)$ such that for every function $w$ satisfying
\begin{align*}
\left\{ \begin{aligned}
         w &\geq 0\qquad &\text{a.e. in } (-2,0) \times \R^d,\\
 \partial_t w - Lw &\geq -\epsilon_0 &\text{in }  (-2,0) \times B_2,\\
 \bet{ D_\ominus \cap \{ w \geq 1\}} &\geq \sigma \bet{D_\ominus},
        \end{aligned} \right.
\end{align*}
the following estimate holds:
\begin{align}
 w \geq \delta \qquad \text{a.e. in } D_\oplus.
\end{align}
The constants $\epsilon_0$ and $\delta$ depend on $\sigma, \alpha_0,\Lambda, d$ but not on $\alpha \in (\alpha_0,2)$.
\end{cor}
\begin{proof}
Application of \autoref{thm:harnack} to $w$ yields
\[ \sigma \leq \fint_{D_\ominus} w(t,x) \d x \d t  \leq c \left( \inf_{D_\oplus} w + \epsilon_0 \right)\]
for a constant $c=c(d,\alpha_0,\Lambda)$. Choosing $\epsilon_0 < \frac{\sigma}{c}$ and $\delta=\frac{\sigma-c \epsilon_0}{c}$ we obtain
\[ \inf_{D_\oplus} w \geq \delta, \]
which is the desired inequality.
\end{proof}

\section{Proof of Hölder regularity estimates}\label{sec:hoelder}
In this section we deduce \autoref{thm:hoelder_main} from \autoref{thm:harnack}. This step is not trivial and differs from the proof in
the case of a local differential operator because the (super-)solutions in \autoref{thm:harnack} are assumed to be nonnegative in the whole spatial domain. Note that the auxiliary functions of the type $M(t,x) = \sup_Q u - u(t,x)$ and $m(t,x)=u - \inf_Q u$ used in \cite[Section 2]{moser-alt} are nonnegative in $Q$ but not in all of $\R^d$. The key idea to overcome this problem is to derive \autoref{lem:w-lemma} from the Harnack inequality. \autoref{lem:w-lemma} then implies \autoref{thm:hoelder_main}. This step is carried out in \cite{Sil06} for elliptic equations.
 
Define for $(t,x) \in \R^{d+1}$ a distance function 
\begin{align*}
 \widehat \rho( (t,x)) = 
 \begin{cases}
  \max\left( \tfrac13 \bet{x}, \tfrac12 {(-t)}^{1/\alpha} \right)\qquad &\text{if } t \in (-2,0],\\
  \infty &\text{if } t \notin (-2,0].
 \end{cases}
\end{align*}
Note that $\widehat \rho((x,t)) \neq \widehat \rho(-(x,t))$. We define 
\[ \widehat D_r((x_0,t_0)) = \left\{ (t,x) \in \R^{d+1}\ \left|\ \widehat \rho((t,x)-(t_0,x_0)) < r \right. \right\},\qquad I_1=(-2,0) \]
and note 
\[\widehat D_r((x_0,t_0)) = (t_0-2r^\alpha, t_0) \times B_{3r}(x_0) \quad \text{and} \quad \bigcup_{r >0} \widehat D_r((0,0))=I_1 \times \R^d.\] 
To simplify notation we write $\widehat D(r) = \widehat D_r((0,0))$. Additionally, we define $D(r)=(-2r^\alpha,0) \times B_{2r}(0)$ and recall the definitions of $D_\oplus$ and $D_\ominus$ in \autoref{cor:silvestre}.
\begin{figure}[ht] \centering
 \includegraphics{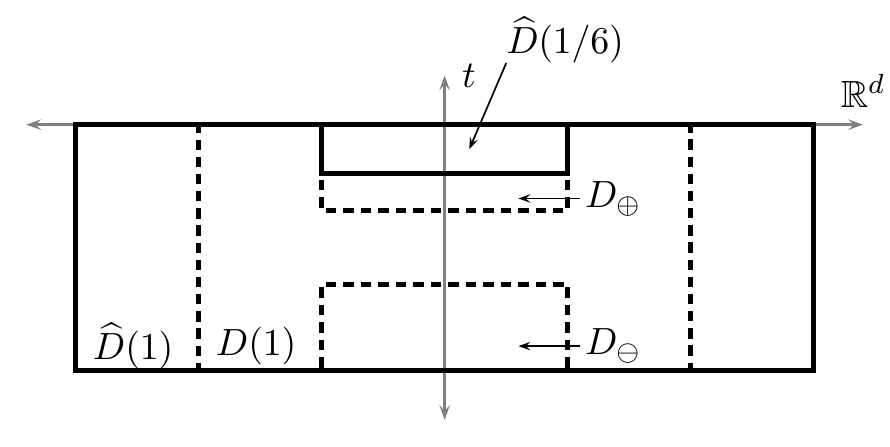}
\end{figure}
\begin{lem} \label{lem:w-lemma} Assume that $L$ is defined by \eqref{eq:def_L} with a kernel $k$ belonging to some $\cK'(\alpha_0,\Lambda)$.
Then there exist ${\beta_0} \in (0,1)$ and $\delta \in (0,1)$ depending on $d,\alpha_0$ and $\Lambda$ such that for every function $w$ with the properties
\begin{subequations}
  \begin{align}
 w &\geq 0 \qquad \text{a.e. in } \widehat D(1), \label{eq:w-a} \\
 \partial_t w - L w &\geq 0 \qquad \text{in } \widehat D(1), \label{eq:w-b} \\
 \bet{ D_\ominus \cap \{ w \geq 1\} } &\geq \frac12 \bet{D_\ominus}, \label{eq:w-c} \\
 w &\geq 2 \left[ 1 - \left(6\, \widehat \rho(t,y)\right)^{\beta_0} \right] \qquad \text{a.e. in } I_1 \times( \R^d \setminus B_3), \label{eq:w-d}
\end{align}
\end{subequations}
the following inequality holds: \[ w \geq \delta \qquad \text{a.e. in } D_\oplus. \] 
\end{lem}
\begin{proof}
The conditions \eqref{eq:w-a} and \eqref{eq:w-b} imply
$ \partial_t w^+ - L w^+ \geq -f$ in $D(1)$, where
\[ f(t,x) = (Lw^-)(t,x) \qquad \text{for } (t,x) \in D(1).\]
Note that since $\bet{x-y} \geq 1$ for $x \in B_2$ and $y \in \R^{d}\setminus B_3$
\begin{align*}
 \norm{f }_{L^\infty(D(1))} = \sup_{(t,x) \in D(1)} \int_{\R^{d}\setminus B_3(0)} w^-(t,y) k_t(x,y) \d y < \infty\ .
\end{align*}
Next, from condition \eqref{eq:w-d} we deduce
\[  w^-(t,y) \leq 2 \left[ 6\ \widehat \rho (t,y)\right]^{\beta_0} - 2 \leq 2 \left( 4^{\beta_0} \bet{y}^{\beta_0} -1 \right)  \qquad \text{a.e. in } I_1 \times (\R^d \setminus B_3).\]
Our aim is to show $\norm{f}_{L^\infty(D(1))} \leq \epsilon_0$ with $\epsilon_0$ as in \autoref{cor:silvestre} for $\sigma=\frac12$. Note that for every $R>3$
\begin{align*}
 \int_{\R^{d}\setminus B_3(0)} \left( 4^{\beta_0} \bet{y}^{\beta_0} -1 \right) k_t(x,y) \d y &= \int_{\R^{d} \setminus B_R(0)} \left( 4^{\beta_0} \bet{y}^{\beta_0} -1 \right) k_t(x,y) \d y \\ &\qquad + \int_{B_R \setminus B_3(0)} \left( 4^{\beta_0} \bet{y}^{\beta_0} -1 \right) k_t(x,y) \d y.
\end{align*}
Because of \eqref{eq:K-3} it is possible to choose $R$ sufficiently large and ${\beta_0} \in (0,1)$ sufficiently small in dependence of $\epsilon_0$ and $\Lambda$ such that $\norm{f}_{L^\infty(D(1))} \leq \epsilon_0$.

Condition \eqref{eq:w-c} ensures that \autoref{cor:silvestre} can be applied.
\end{proof}

\begin{sa}[Oscillation decay] \label{thm:osc-decay} Assume that $L$ is defined by \eqref{eq:def_L} with a kernel $k$ belonging to some $\cK'(\alpha_0,\Lambda)$. Then there exists $\beta \in (0,1)$ depending on $d,\alpha_0$ and $\Lambda$ such that every solution $u$ to $\partial_t u - L u =0$ in $\widehat D(1)$ satisfies for all $\nu \in \Z$
\begin{align} 
 \osc_{\widehat D(6^{-\nu})} u \leq 2\|u\|_{L^\infty(I_1 \times \R^d)} 6^{-\nu\beta}, \label{eq:osc}
\end{align}
where $\osc_Q u = \sup_Q u - \inf_Q u$.
\end{sa}
\begin{proof} Set $K = M_0 - m_0$ where $M_0 = \sup_{I_1 \times \R^{d}} u$, $m_0 = \inf_{I_1 \times \R^{d}} u$. Let $\delta, \beta_0 \in (0,1)$ be the constants from \autoref{lem:w-lemma}. Define 
\begin{equation}\label{eq:beta}
 \beta = \min\Bigl( \beta_0, \frac{\log (\tfrac{2}{2-\delta})}{\log 6} \Bigr)\quad \Longrightarrow \quad 1 - \frac{\delta}{2} < 6^{-\beta}.
\end{equation}
We will construct inductively an increasing sequence $(m_\nu)_{\nu \in \Z}$ and a decreasing sequence\linebreak $(M_\nu)_{\nu \in \Z}$ such that for every $\nu \in \Z$ 
\begin{align}
 \begin{aligned}
  m_\nu \leq u &\leq M_\nu \qquad \text{a.e. in } \widehat D(6^{-\nu}),\\
  M_\nu - m_\nu &= K 6^{-\nu\beta}.
 \end{aligned} 
\label{eq:maximum-minimum}
\end{align}
Obviously, \eqref{eq:maximum-minimum} implies \eqref{eq:osc}. For $n \in \N$ set $M_{-n} = M_0$, $m_{-n}=m_0$.
Assume we have constructed $M_n$ and $m_n$ for $n \leq k-1$ and define
\[ v(t,x)= \left[ u \left(\frac{t}{6^{\alpha(k-1)}} ,\frac{x}{6^{k-1}}  \right)
- \frac{M_{k-1}+m_{k-1}}{2} \right] \frac{2\cdot 6^{\beta(k-1)}}{K}\,.\]
Clearly, $v$ satisfies
\begin{align} \label{eq:v-prop}
 \partial_t v - L v =0 \text{ in } \widehat D(1)\quad \text{ and }\quad \bet{v} \leq 1 \text{ in }\widehat D(1) \text{ (by induction hypothesis)}.
\end{align}
On $I_1 \times (\R^d \setminus B_3)$ we can estimate $v$ in the following way: For $(t,y) \in I_1 \times (\R^d \setminus B_3)$ fix $j \in \N$ such that
\[ 6^{j-1} \leq \widehat \rho(t,y) < 6^{j}, \quad \text{or equivalently } (t,y) \in \widehat D(6^{j}) \setminus \widehat D(6^{j-1}).\] 
Then
\allowdisplaybreaks
\begin{align} 
\frac{K}{2 \cdot 6^{(k-1)\beta}} v(t,y) &= \left( u\left( \frac{t}{6^{\alpha(k-1)}},\frac{y}{6^{k-1}}\right) - \frac{M_{k-1}+m_{k-1}}{2} \right) \nonumber \\
&\leq \left( M_{k-j-1} - m_{k-j-1} + m_{k-j-1} - \frac{M_{k-1}+m_{k-1}}{2} \right) \nonumber\\
&\leq \left( M_{k-j-1} - m_{k-j-1} - \frac{M_{k-1}-m_{k-1}}{2} \right) \nonumber\\
&\leq \left( K 6^{-(k-j-1)\beta} - \frac{K}{2} 6^{-(k-1)\beta} \right),\nonumber \\
\Rightarrow v(t,y) & \leq 2 \cdot 6^{j\beta} -1\qquad \text{for a.e. } (t,y) \in \widehat D(6^{j}) \setminus \widehat D(6^{j-1}) \nonumber \\
\Rightarrow v(t,y) & \leq 2 \left[ 6\ \widehat \rho(t,y)\right]^\beta - 1 \qquad \text{for a.e. } (t,y) \in I_1 \times (\R^d \setminus B_3) . \label{eq:v-kompl1}
\intertext{Analogously, we can estimate $v$ from below by}
 v(t,y) &\geq 1 - 2 \left[ 6\ \widehat \rho(t,y)\right]^\beta \qquad \text{for a.e. } (t,y) \in I_1 \times (\R^d \setminus B_3). \label{eq:v-kompl2}
\end{align}
Now there are two cases. In the first case $v$ is non-positive in at least half of the set $D_\ominus$, i.e.
\begin{equation} \label{eq:half}
 \bet{ D_\ominus \cap \{ v \leq 0\} } \geq \frac12 \bet{D_\ominus}.
\end{equation}
Set $w=1-v$. $w$ satisfies conditions \eqref{eq:w-a}-\eqref{eq:w-d} of \autoref{lem:w-lemma} and hence
\[ w \geq \delta \quad \text{a.e. in } D_\oplus, \quad \text{ or equivalently} \quad v \leq 1 - \delta \quad \text{a.e. in } D_\oplus.\]

Noting that $\widehat D(1/6) \subset D_\oplus$ this estimate has the following consequence for $u$: For a.e. $(t,x) \in \widehat D(6^{-k})$ we have
\begin{align*}
u(t,x) &=  \frac{K}{2 \cdot 6^{(k-1)\beta}}\ v\left(6^{\alpha (k-1)} t ,{6^{k-1} x}  \right) + \frac{M_{k-1}+m_{k-1}}{2} \\
&\leq \frac{K (1-\delta)}{2 \cdot 6^{(k-1)\beta}} + m_{k-1} + \frac{M_{k-1}-m_{k-1}}{2} \\
&\leq \frac{K (1-\delta)}{2 \cdot 6^{(k-1)\beta}} + m_{k-1} + \frac{K}{2 \cdot 6^{(k-1)\beta}} = m_{k-1} + \left(1-\frac{\delta}{2} \right) K 6^{-(k-1)\beta} \\
&\leq m_{k-1} + K 6^{-k\beta}, 
\end{align*}
where we apply \eqref{eq:beta} in the last inequality.
By choosing $m_k = m_{k-1}$ and $M_k = m_{k-1} + K 6^{-k\beta}$ we obtain sequences $(m_n)$ and $(M_n)$ satisfying \eqref{eq:maximum-minimum}.
In the second case $v$ is positive in at least half of the set $D_\ominus$ and hence $w=1+v$ satisfies all conditions of \autoref{lem:w-lemma}. Therefore, we obtain
\[ w \geq \delta \quad \text{a.e. in } D_\oplus, \quad \text{ or equivalently} \quad v \geq -1 + \delta  \quad \text{a.e. in } D_\oplus.\]
Adopting the computations above we see that $M_k = M_{k-1}$ and $m_k = M_{k-1} - K 6^{-k\beta}$ lead to the desired result.

This proves \eqref{eq:maximum-minimum}.
\end{proof}

Having established \autoref{thm:osc-decay} we are now able to prove \autoref{thm:hoelder_main} providing a priori estimates of Hölder norms of solutions.

\begin{proof}[Proof of \autoref{thm:hoelder_main}]
Let $u$ as in \autoref{thm:hoelder_main}, $Q' \Subset Q$ and define
\[ \eta(Q',Q)=\eta = \sup \left\{ r\in (0,\tfrac12] \ \left|\ \forall (t,x) \in Q' \colon \widehat D_r(t,x) \subset Q \right. \right\}.\]
Fix $(t,x),(s,y) \in Q'$. Without loss of generality $t\leq s$. At first, assume that 
\begin{equation} \label{eq:rho}
 \widehat \rho((t,x)-(s,y))<\eta 
\end{equation}
and choose $n \in \N_0$ such that
\[ \frac{\eta}{6^{n+1}} \leq \widehat \rho((t,x)-(s,y)) < \frac{\eta}{6^n}. \]
Now set $\overline u(t,x)=u(\eta^\alpha t + s, \eta x + y)$. By assumption $\overline u$ is a solution of $\partial_t \overline u - L \overline u = 0$ in $\widehat D(1)$. Accordingly, applying \autoref{thm:osc-decay} to $\overline u $ we obtain
\begin{align*}
\bet{u(t,x)-u(s,y)}& = \bet{\overline u(\eta^{-\alpha}(t-s), \eta^{-1}(x-y))- \overline u(0,0)} \\
&\leq 2 \norm{\overline u}_{L^\infty(I_1 \times \R^d)} 6^{-n\beta} \leq 2 \norm{ u}_{L^\infty(I \times \R^d)} \left( 6^{-n-1}\right)^\beta  6^{\beta} \\
&\leq 12 \norm{ u}_{L^\infty(I \times \R^d)} \left( \frac{ \widehat \rho( (t,x)-(s,y) )}{\eta} \right)^\beta \\
&\leq 12 \norm{ u}_{L^\infty(I \times \R^d)} \left( \frac{\bet{x-y} + (s-t)^{1/\alpha}}{\eta}   \right)^\beta .
\end{align*}
Hence, for all $(t,x),(y,s) \in Q'$ subject to \eqref{eq:rho}
\begin{align*}
 \frac{ \bet{u(t,x)-u(s,y)} }{\left( \bet{x-y} + \bet{t-s}^{1/\alpha} \right)^{\beta}} \leq \frac{12 \norm{u}_{L^\infty(I \times \R^d)}}{\eta^\beta}\ .
\end{align*}
If $\widehat \rho((t,x)-(s,y))\geq \eta$ then the Hölder estimate follows directly:
\begin{align*}
 \bet{u(t,x)-u(s,y)} &\leq 2 \norm{u}_{L^\infty(I \times \R^d)} \leq  \frac{2 \norm{u}_{L^\infty(I \times \R^d)} \left[\max\left(\bet{x-y},\bet{t-s}^{1/\alpha}\right)\right]^\beta}{\eta^\beta}\\
 &\leq \frac{2 \norm{u}_{L^\infty(I \times \R^d)}}{\eta^\beta} \left(\bet{x-y}+\bet{t-s}^{1/\alpha}\right)^\beta.
\end{align*}
Hence,
\begin{align*}
 \sup_{(t,x),(s,y) \in Q'} \frac{ \bet{u(t,x)-u(s,y)} }{\left( \bet{x-y} + \bet{t-s}^{1/\alpha} \right)^{\beta}} \leq \frac{12 \norm{u}_{L^\infty(I \times \R^d)}}{\eta^\beta} ,
\end{align*}
which had to be shown.
\end{proof}

\begin{appendix}
\section{Steklov averages} \label{sec:A-1}

The aim of this appendix is to justify the use of \eqref{eq:steklov-solution} instead of \eqref{eq:def-solution} in our main technical results, \autoref{prop:step-neg}, \autoref{prop:step-pos} and \autoref{prop:loglemma}. Thus, we can work with supersolutions $u$ as if they were a.e. differentiable with respect to $t$. This approach is standard when proving regularity results for solutions to parabolic problems, cf. \cite[Sec. 9]{ArSe67}. Nevertheless, we provide details and show that the nonlocality (in space) of our parabolic operator does not form a serious obstacle.

In the above mentioned proofs we multiply \eqref{eq:steklov-solution} with some piecewise differentiable function $\chi \colon \R \to [0,\infty)$ and integrate over some time interval $(t_1,t_2) \subset I$. This implies, together with the chain rule and partial integration,
\begin{multline}\label{eq:A-1}
 \left[ \chi(t) \int_{\Omega'} \psi(x) w(t,x) \d x \right]_{t=t_1}^{t_2}   + \int_{t_1}^{t_2} \chi(t) \cE_t(u,\phi) \d t\\ \geq \int_{t_1}^{t_2} \chi (t) \int_{\Omega'} f(t,x) \phi(x) \d x \d t + \int_{t_1}^{t_2} \chi'(t) \int_{\Omega'} \psi(x) w(t,x) \d x \d t,
\end{multline}
where 
\begin{equation}\label{eq:def-w}
 w(t,x)=\begin{cases}
         \frac{1}{1-q} u^{1-q}(t,x) \qquad &\text{if } q \neq 1,\\
         \log u(t,x) &\text{if } q =1.
        \end{cases}
\end{equation}
Inequality \eqref{eq:A-1} is the main source for our estimates. Let us now show how to derive \eqref{eq:A-1} from \eqref{eq:def-solution}. To this end, we introduce the concept of
\emph{Steklov averages} (cf. \cite{diBenedetto-neu},
\cite{lady}): Let $I=(T_1,T_2)$, $Q=I \times \Omega$. For $v \in L^1(Q)$ and $0<h<T_2-T_1$ define
\[ v_h(\cdot,t) = \begin{cases}
                    \displaystyle \frac1h \int_{t}^{t+h} v(\cdot,s)\d s \quad
&\text{for } T_1 < t<T_2 -h,\\
		    0, &\text{for } t \geq T_2-h    .               \end{cases}\]

Fix $t \in I$, $\Omega' \Subset \Omega$ and $h>0$ such that $t+h \in I$. In \eqref{eq:def-solution} we choose $\phi(s,x)=\eta(x)$ with $\eta \in H_0^{\alpha/2}(\Omega')$, $t_1=t$ and $t_2=t+h$. Dividing by $h$ we obtain
 \begin{align}\label{eq:cons_sol}
 \int_{\Omega'} \partial_t u_h(t,x) \eta(x) \d x + \frac1h \int_{t}^{t+h} \cE_s(u(s,\cdot),\eta(\cdot)) \d s \geq \int_{\Omega'} f_h(t,x) \eta(x) \d x \ , 
\end{align}
valid for all $t \in I$ and $\eta \in H_0^{\alpha/2}(\Omega')$.

Next we choose in \eqref{eq:cons_sol} (for fixed $t \in I$) test functions of the form $\eta=\chi(t)\psi u_h^{-q}(t,\cdot)$, $q>0$, where $\psi,\chi$ are suitable cut-off functions. Then we integrate \eqref{eq:cons_sol} over some time interval $(t_1,t_2)$. Hence, with $w$ as in \eqref{eq:def-w},
\begin{multline*}
 \int_{t_1}^{t_2} \chi(t) \int_{\Omega'} \psi(x) \partial_t w_h(t,x) \d x \d t + \int_{t_1}^{t_2} \chi(t) \frac1h \int_{t}^{t+h} \cE_s(u(s,\cdot),\psi(\cdot)u_h^{-q}(t,\cdot)) \d s \d t\\
  \geq \int_{t_1}^{t_2} \chi(t) \int_{\Omega'} f_h(t,x) \psi(x) u_h^{-q}(t,x) \d x \d t \, .
\end{multline*}
After partial integration in the first term we pass to the limit $h \to 0$. \autoref{lem:steklov} and \autoref{lem:appenix-convergence} from below then imply
\begin{multline}
  \left[\chi(t) \int_{\Omega'}\psi(x) w(t,x) \d x \right]_{t=t_1}^{t_2} + \int_{t_1}^{t_2} \chi(t) \cE_t(u(t,\cdot),\psi(\cdot)u^{-q}(t,\cdot)) \d t \\ \geq 
  \int_{t_1}^{t_2} \chi(t) \int_{\Omega'} f(t,x) u(t,x) \psi(x) \d x \d t + \int_{t_1}^{t_2} \chi'(t) \psi(x) w(t,x) \d x \d t\, ,
\end{multline}
which we wanted to show. 

It remains to prove two auxiliary results.

\begin{lem} \label{lem:steklov}
Let $X$ be a Banach space and let $v \in L^{p}(I;X)$ for some $p \in [1,\infty]$ and $I'=(t_1,t_2) \subset I$ with $t_2<T_2$. Then
\begin{enumerate}[(i)]\itemsep0.5em
 \item $v_h \in C(\overline {I'};X)$ for every $h \in (0,T_2-t_2)$.
 \item $\norm{v_h}_{L^p(I';X)} \leq \norm{v}_{L^p(I';X)}$ for every $h \in
(0,T_2-t_2)$. 
 \item If $p<\infty \colon \norm{v_h-v}_{L^p(I';X)} \to 0 \quad \text{for } h \searrow 0$.
 \item If $v\in C(I,L^p(\Omega))$ for some $p<\infty$, then $\norm{v_h(t,\cdot)-v(t,\cdot)}_{L^p(\Omega)} \xrightarrow{h\to0+} 0$ for every $t \in I'$.
\end{enumerate}\end{lem}
The proof of this lemma is quite elementary; some of the assertions above are proved in \cite[II.§4]{lady}.
\begin{lem}\label{lem:appenix-convergence}
Let $v$ be a positive supersolution to $\partial_t v - Lv=f$ in $Q=\Omega \times I$. Let $\phi$ be an admissible test function as in \autoref{defi:solution} that is bounded and satisfies $\supp[\phi(t,\cdot)] \subset B_R \Subset \Omega$ for some $R>0$ and every $t \in I$. Then for every $I' \Subset I$
\begin{equation}\label{eq:convergence}
 \int_{I'} \frac1h \int_{t}^{t+h} \cE_s(v(s,\cdot),\phi_h(t,\cdot)) \d s \d t 
\xrightarrow{h \to 0+} \int_{I'} \cE_t(v,\phi) \d t .
\end{equation}
\end{lem}
\begin{proof} Set $V(t,x,y)=a(t,x,y)\left(v(t,x)-v(t,y)\right)$ and $\Phi(t,x,y)=\phi(t,x)-\phi(t,y)$. Since
\begin{align*}
 \int_{I'} \frac1h \int_{t}^{t+h} \cE_s(v(s,\cdot),\phi_h(t,\cdot)) \d s \d t 
 &= \int_{I'} \iint\limits_{\R^d\, \R^d} V_h(t,x,y) \Phi_h(t,x,y) k_0(x,y) \d x \d y \d t\ , \\
\int_{I'} \cE_t(v,\phi) \d t &= \int_{I'} \iint\limits_{\R^d\, \R^d} V(t,x,y) \Phi(t,x,y) k_t(x,y) \d x \d y \d t\ ,
\end{align*}
the convergence in \eqref{eq:convergence} follows if we show both
\begin{subequations}
\begin{align}
 \int_{I'} \iint\limits_{\R^d\, \R^d} \bet{V_h(t,x,y)-V(t,x,y)} \bet{\Phi(t,x,y)} k_0(x,y) \d x \d y \d t &\xrightarrow{h \to 0+} 0 , \label{eq:teilA}\\
 \int_{I'} \iint\limits_{\R^d\, \R^d} \bet{V_h(t,x,y)}  \bet{\Phi_h(t,x,y)-\Phi(t,x,y)} k_0(x,y) \d x\d y \d t & \xrightarrow{h \to 0+} 0. \label{eq:teilB}
\end{align} 
\end{subequations}
First we prove \eqref{eq:teilA}. Define $B=B_{R+\epsilon}$ for some fixed $\epsilon>0$. A usual decomposition of the integral over $\R^d \times \R^d$ yields
\begin{align*}
 \int_{I'} \iint\limits_{\R^d\, \R^d} & \bet{V_h(t,x,y)-V(t,x,y)} \bet{\phi(t,x)-\phi(t,y)} k_0(x,y) \d x \d y \d t \\
 &= \int_{I'} \iint\limits_{B\, B} \bet{V_h(t,x,y)-V(t,x,y)} \bet{\phi(t,x)-\phi(t,y)} k_0(x,y) \d x \d y \d t \\
 &\qquad + 2 \int_{I'} \int_{B} \bet{\phi(t,x)} \int_{B^c} \bet{V_h(t,x,y)-V(t,x,y)} k_0(x,y)\d y\d x \d t \\&=: I_1(h) + I_2(h).
\end{align*}
Hölder's inequality applied to $I_1(h)$ shows that $I_1(h) \to 0$:
\enlargethispage{1em}
\begin{align*}
 \norm{(V_h-V) \Phi_h\, k_0}_{L^1(I' \times B \times B)} &\leq \norm{(V_h-V) {k_0}^{\frac12}}_{L^2(I';L^2(B\times B))} \norm{\Phi {k_0}^{\frac12}}_{L^2(I';L^2(B \times B))} \\
&\leq  \norm{(V_h-V) {k_0}^{\frac12}}_{L^2(I';L^2(B\times B))} \norm{\phi}_{L^2(I';H^{\alpha/2}(B))},
\end{align*}
where we have used \eqref{eq:K-2} in the second inequality. \autoref{lem:steklov}(iii) implies that the first factor tends to zero since -- again due to property \eqref{eq:K-2} --
\[ v \in L^2(I';H^{\alpha/2}(B)) \Rightarrow V\, k_0^{\frac12} \in L^2(I';L^2(B \times B)).\]
In a similar way we obtain the convergence of $I_2(h)$:
\begin{align*}
\int_{I'} \int_{B}& \bet{\phi(t,x)} \int_{B^c} \bet{V_h(t,x,y)-V(t,x,y)} k_0(x,y)\d y\d x \d t \\
&\leq \norm{\phi}_{L^\infty(I'\times B)} \int_{I'} \norm{V_h(t,\cdot,\cdot)-V(t,\cdot,\cdot)}_{L^\infty(\R^d \times \R^d)} \iint\limits_{B_R\, B^c} k_0(x,y) \d y \d x \d t\\
&\leq \Lambda\, \epsilon^{-\alpha}\, \bet{B_R}\, \norm{\phi}_{L^\infty(I'\times B)} \norm{V_h-V}_{L^1(I;L^\infty(\R^d \times \R^d))},
\end{align*}
where we have applied \eqref{eq:K-1} in the second inequality. The convergence of the last factor follows again from \autoref{lem:steklov}(iii).

Next, we prove \eqref{eq:teilB}: Again, we use the decomposition
\begin{align*}
 \int_{I'} \iint\limits_{\R^d\, \R^d} & \bet{V_h(t,x,y)}  \bet{\Phi_h(t,x,y)-\Phi(t,x,y)} k_0(x,y) \d x\d y \d t \\
 &= \int_{I'} \iint\limits_{B\, B} \bet{V_h(t,x,y)} \bet{\Phi_h(t,x,y)-\Phi(t,x,y)} k_0(x,y)\d x \d y \d t \\
 &\qquad + 2 \int_{I'} \int_{B} \bet{\phi_h(t,x)-\phi(t,x)} \int_{B^c} \bet{V_h(t,x,y)} k_0(x,y) \d y \d x \d t\\
 &=: J_1(h) + J_2(h).
\end{align*}
The convergence $J_1(h) \xrightarrow{h \to 0+} 0$ follows by the same argument as we used for the convergence of $I_1(h)$. It remains to show that $J_2(h) \xrightarrow{h \to 0+} 0$:
\begin{align*}
 \tfrac12 J_2(h) &\leq \int_{I'} \norm{\phi_h(t, \cdot) - \phi(t,\cdot)}_{L^\infty(B_R)} \iint\limits_{B_R\, B^c} \bet{V_h(t,x,y)} k_0(x,y) \d x \d y \d t \\
 &\leq 2 \epsilon^{-\alpha} \bet{B_R} \norm{v_h}_{L^\infty(I';L^\infty(\R^d))}  \norm{\phi_h-\phi}_{L^1(I';L^\infty(\R^d))}.
\end{align*}
Finally, we apply \autoref{lem:steklov}(ii),(iii) to conclude that $J_2(h)$ converges to zero. This finishes the proof.
\end{proof}
\end{appendix}

\bibliography{lit_harnack}
\bibliographystyle{bibstyle_english_num}
\end{document}